\documentclass{amsart}

\usepackage{amssymb}
\usepackage{eepic}
\usepackage{color}
\usepackage{mathtools,xparse}
\usepackage[linktocpage=true]{hyperref}
\usepackage[utf8x]{inputenc}
\allowdisplaybreaks

\usepackage{comment}
\usepackage[usenames,dvipsnames,table]{xcolor}
\usepackage{tikz}

\newtheorem{thm}{Theorem}[section]
\newtheorem{cor}[thm]{Corollary}
\newtheorem{lem}[thm]{Lemma}

\newtheorem{rem}[thm]{Remark}
\newtheorem{example}[thm]{Example}
\theoremstyle{definition}
\newtheorem{defn}[thm]{Definition}

\newtheorem{TheoA}{\bf Theorem A}
\newtheorem{TheoB}{\bf Theorem B}
\newtheorem{CorB}{\bf Corollary B1}
\newtheorem{CorBB}{\bf Corollary B2}

\numberwithin{equation}{section}

\newcommand{\A}{\Omega}
\newcommand{\Z}{\mathbf{Z}}

\newcommand{\cU}{\mathcal{U}}
\newcommand{\N}{\mathbf{N}}

\newcommand{\R}{\mathbf{R}}

\newcommand{\C}{\mathbf{C}}
\newcommand{\SL}{\mathrm{SL}}
\newcommand{\SO}{\mathrm{SO}}
\newcommand{\PSL}{\mathrm{PSL}}

\newcommand{\manifold}{\Pi}
\DeclareMathOperator{\Tr}{Tr}
\DeclareMathOperator{\tr}{tr}
\DeclareMathOperator{\Ad}{Ad}
\DeclareMathOperator{\Aff}{Aff}
\DeclareMathOperator{\ad}{ad}

\newcommand{\vect}{\operatorname{span}}

\newcommand{\diff}{\mathrm{Diff}}

\newcommand{\cL}{\mathcal{L}}

\def\G{\mathrm{G}}
\def\1{\mathbf{1}}

\newcommand{\GL}{\textrm{GL}}

\newcommand{\norm}[2]{\left\|#1\right\|_{#2}}

\newcommand{\fin}{\hspace*{\fill} $\square$ \vskip0.2cm}

\renewcommand{\thefootnote}{\dag}

\begin{document}
\null

\vskip-25pt

\null

\begin{center}
{\huge The local geometry of \\ idempotent Schur multipliers}

\vskip7pt

{\sc {Javier Parcet, Mikael de la Salle and Eduardo Tablate}}
\end{center}

\title[Idempotent Schur multipliers]{}

\author[J. Parcet, M. de la Salle, E. Tablate]{}

\maketitle

\null

\vskip-58pt

\null

\begin{center}
{\large {\bf Abstract}}
\end{center}

\vskip-28pt

\null

\begin{abstract}
A Schur multiplier is a linear map on matrices which acts on its entries by multiplication with some function, called the symbol. We consider idempotent Schur multipliers, whose symbols are indicator functions of smooth Euclidean domains. Given $1<p\neq 2<\infty$, we provide a local characterization (under some mild transversality condition) for the boundedness on Schatten $p$-classes of Schur idempotents in terms of a lax notion of boundary flatness. We prove in particular that all Schur idempotents are modeled on a single fundamental example: the triangular projection. As an application, we fully characterize the local $L_p$-boundedness of smooth Fourier idempotents on connected Lie groups. They are all modeled on one of three fundamental examples: the classical Hilbert transform, and two new examples of Hilbert transforms that we call affine and projective. Our results in this paper are vast noncommutative generalizations of Fefferman's celebrated ball multiplier theorem. They confirm the intuition that Schur multipliers share profound similarities with Euclidean Fourier multipliers |even in the lack of a Fourier transform connection| and complete, for Lie groups, a longstanding search of Fourier $L_p$-idempotents.
\end{abstract}

\maketitle

\addtolength{\parskip}{+1ex}

\section*{\bf \large Introduction}

Schur multipliers are linear maps on matrix algebras with a great impact on geometric group theory, operator algebras, and functional analysis. Their definition is rather simple on discrete spaces $S_M(A) = (M(j,k) A_{jk})_{jk}$. It easily extends to nonatomic $\sigma$-finite measure spaces $(\Omega,\mu)$, by restricting to operators $A$ in $L_2(\Omega,\mu)$ admitting a kernel representation over $\Omega \times \Omega$. Their role in geometric group theory and operator algebras was first analyzed by Haagerup. His pioneering work on free groups \cite{H} and the research thereafter on semisimple lattices \cite{DCH,CH} encoded deep geometric properties via approximation properties with Schur multipliers. Other interesting links can be found in \cite{Bennett, PisAJM, PisSim, PisBAMS, PisS, PS}. 

In 2011, stronger rigidity properties of high rank lattices were discovered by studying $L_p$-approximations \cite{dLdlS,LdlS}. In first place, there are no $L_p$-approximations by means of Fourier or Schur multipliers over $\SL_n(\R)$ for $p > 2+ \alpha_n$, with $\alpha_n \to 0$ as $n \to \infty$. Secondly, it turns out that this unprecedented pathology leads to a strong form of nonamenability which is potentially useful to distinguish the group von Neumann algebras of $\PSL_n(\Z)$ for different values of $n \ge 3$, the most iconic form of Connes' rigidity conjecture. This has strongly motivated our recent work \cite{CGPT1,PRS} with several forms of the H\"ormander-Mikhlin theorem. Nevertheless, there is still much to learn about less regular multipliers. A key point in \cite{LdlS} was a careful analysis of Schur multipliers over the $n$-sphere for symbols of the form $M_\varphi(x,y) = \varphi(\langle x,y \rangle)$. More precisely, the boundedness of $S_{M_\varphi}$ on the Schatten class $S_p$ for $p > 2 + \frac{2}{n-1}$ imposes H\"older regularity conditions on $\varphi$. This article grew from the analysis of the \emph{spherical Hilbert transform} 
\begin{equation*}\label{eq:sphericalHilbertTransform}
    H_\mathbf{S}: A \mapsto \Big( \mathrm{sgn} \langle x,y \rangle A_{xy} \Big)_{xy}.
\end{equation*} 

Is it $S_p$-bounded for some $\frac{2n}{n+1} < p \neq 2 < \frac{2n}{n-1}$? It is worth noting the analogy with the ball multiplier problem, which was only known to be unbounded for $p$ outside this range before Fefferman's celebrated contribution \cite{Fe}. Our main result completely solves this problem: $H_\mathbf{S}$ is $S_p$-unbounded unless $n=1$ or $p=2$. In fact we characterize $S_p$-boundedness for a much larger class of idempotents. 

Let $M,N$ be two differentiable manifolds with the Lebesgue measure coming from any Riemmanian structure on them. Consider a $\mathcal{C}^1$-domain $\Sigma \subset M \times N$ so that its boundary $\partial \Sigma$ is a smooth hypersurface, which is locally represented by level sets of some real-valued $\mathcal{C}^1$-functions with nonvanishing gradients. We say that $\partial \Sigma$ is transverse at a point $(x,y)$ when the tangent space of $\partial \Sigma$ at $(x,y)$ maps surjectively on each factor $T_x M$ and $T_y N$. In that case, both sections $$\partial \Sigma_x = \big\{ y' \in N \mid (x,y') \in \partial \Sigma \big\} \quad \mbox{and} \quad \partial \Sigma^y = \big\{ x' \in M \mid (x',y) \in \partial \Sigma \big\}$$ become codimension $1$ manifolds on some neighbourhood of $y$ and $x$ respectively. 

\begin{TheoA} \label{ThmA} 
\emph{Let $p \in (1,\infty) \setminus \{2\}$ and consider a $\mathcal{C}^1$-domain $\Sigma \subset M \times N$. Then the following statements are equivalent for any transverse point $(x_0,y_0) \in \partial \Sigma \hskip-2pt :$}  
\begin{enumerate}
\item\label{item:Spbounded}  $S_p$-boundedness. \emph{The idempotent Schur multiplier $S_\Sigma$ whose symbol equals $1$ on $\Sigma$ and $0$ elsewhere is bounded on $S_p(L_2(U), L_2(V))$ for some pair of neighbourhoods $U,V$ of $x_0, y_0$ in $M,N$.}

\vskip2pt

\item\label{item:SteinCurvature} Zero-curvature condition. \emph{There are neighbourhoods $U,V$ of $x_0, y_0$ in $M,N$ such that the tangent spaces $T_y (\partial \Sigma_{x_1})$ and $T_y (\partial \Sigma_{x_2})$ coincide for any pair of points $(x_1,y), (x_2,y) \in \partial \Sigma \cap (U \times V)$.} 

\vskip2pt

\item\label{item:HilbertTransform} Triangular truncation representation. \emph{There are neighbourhoods $U,V$ of the points $x_0, y_0$ in $M,N$ and $\mathcal{C}^1$-functions $f_1: U \to \R$ and $f_2: V \to \R$, such that $\Sigma \cap (U \times V) = \big\{(x,y) \in U \times V: f_1(x)>f_2(y) \big\}$.}
\end{enumerate}
\end{TheoA} 

Theorem A characterizes the local geometry of $S_p$-bounded idempotent Schur multipliers and vastly amplifies the ball multiplier theorem \cite{Fe}. A first interesting consequence is that this property does not depend on the value of $p$. The implication \eqref{item:HilbertTransform}$\Rightarrow$\eqref{item:Spbounded} is the easy one. Using known techniques for Schur multipliers, it follows from the classical $S_p$-boundedness of the triangular projection $(A_{jk}) \mapsto (\chi_{j \geq k} A_{jk})$ with $j,k \in \N$, closely related to the $L_p$-boundedness of the Hilbert transform. On the contrary, the converse implication \eqref{item:Spbounded}$\Rightarrow$\eqref{item:HilbertTransform} is certainly unexpected. It says that the triangular projection is the only local model for $S_p$-bounded idempotent multipliers. Our proof decomposes in two completely independent parts. First the implication \eqref{item:Spbounded}$\Rightarrow$\eqref{item:SteinCurvature} is very much analytical. Due to well-known Fourier-Schur transference results for Toeplitz symbols, Fefferman's theorem corresponds to the special case $M=N=\R^n$ and domains of the form $\Sigma = \big\{ (x,y): x-y \in \Omega \big\}$ for a Euclidean $\mathcal{C}^1$-domain $\Omega$. Transversality trivially holds at every boundary point in this case. In the general case, the main idea |Lemma \ref{prop:Meyer_Lemma_NC}| is a new connection between Schur and Fourier multipliers: it gives an $L_p$ square-function inequality for directional Hilbert transforms out of the $S_p$-boundedness of $S_\Sigma$. This is a noncommutative form of Meyer's lemma, which derived such a square inequality from $L_p$-bounded Fourier multipliers, and which was a key part in the proof of the ball multiplier theorem. On the other hand, the implication \eqref{item:SteinCurvature}$\Rightarrow$\eqref{item:HilbertTransform} is a purely geometric statement about hypersurfaces in product manifolds |Theorem~\ref{thm:hypersurfaces_in_products}| to which we failed to find a straightforward proof.  

It is rather surprising to us that Theorem A holds for Schur multipliers on general manifolds which |contrary to Euclidean spaces, where Fefferman's result held so far| lack to admit a Fourier transform connection. Also observe that, when we take $M=N=\R^n$ and  write $$\mathbf{n}(x,y)  = \big( \mathbf{n}_1(x,y),\mathbf{n}_2(x,y) \big)$$ for a normal vector to $\partial \Sigma$ at $(x,y)$, transversality means that both $n$-dimensional components $\mathbf{n}_1, \mathbf{n}_2$ are nonzero. The zero-curvature condition means that $\mathbf{n}_2(x_1,y)$ and $\mathbf{n}_2(x_2,y)$ are parallel|equivalent forms in terms of $\mathbf{n}_1(x,y_1)$ and $\mathbf{n}_1(x,y_2)$ instead, or simpler formulations for $\mathcal{C}^2$-domains will be also discussed. In a different direction, a global (nonlocal) characterization of $S_p$-bounded idempotent Schur multipliers also follows for relatively compact fully transverse domains $\Sigma$. 

\begin{figure}
  \begin{tikzpicture}[scale=1.7]
    \newcommand\exc{.25};
    \newcommand\ccc{.6};
    \newcommand\rr{.7};
    \newcommand\rrr{.65};
    \newcommand\delt{.93};
    \newcommand\factor{.5};
    \draw (0,0) circle (1);
    \draw[dashed,gray,very thin] (1,0) arc(0:180:1 and \exc);
    \draw[gray,very thin] (1,0) arc(360:180:1 and \exc);
    \begin{scope}[rotate=15]
    \pgfmathparse{\rr}\let\rayon\pgfmathresult;
    \pgfmathparse{sqrt(1-\rayon*\rayon)}\let\y\pgfmathresult;
    \pgfmathparse{sqrt(1-\delt*\delt)}\let\z\pgfmathresult;
    \pgfmathparse{asin((\delt*\y - sqrt(1-\ccc))/(\rayon*\z))}\let\theta\pgfmathresult;
    \pgfmathparse{cos(\theta)}\let\c\pgfmathresult;
    \pgfmathparse{sin(\theta)}\let\s\pgfmathresult;
    \pgfmathparse{atan(\z*\c/(\delt*\rayon+\y*\z*\s))}\let\alpha\pgfmathresult;
    \begin{scope}[rotate=\alpha,color=red]
      \draw[dashed,->] (0,0)--(0,\delt*\rayon);  
      \draw (0,\delt*\rayon) ellipse ({\z} and {\y*\z});
      \draw (-\z,\delt*\rayon);
      \draw[shift={(\z*\c,\delt*\rayon+\y*\z*\s)},scale=\factor,->] (0,0)--(-\s,\c*\y);
    \end{scope}

    \pgfmathparse{\rrr}\let\rayon\pgfmathresult;
    \pgfmathparse{sqrt(1-\rayon*\rayon)}\let\y\pgfmathresult;
    \pgfmathparse{sqrt(1-\delt*\delt)}\let\z\pgfmathresult;
    \pgfmathparse{180-asin((\delt*\y - sqrt(1-\ccc))/(\rayon*\z))}\let\theta\pgfmathresult;
    \pgfmathparse{cos(\theta)}\let\c\pgfmathresult;
    \pgfmathparse{sin(\theta)}\let\s\pgfmathresult;
    \pgfmathparse{atan(\z*\c/(\delt*\rayon+\y*\z*\s))}\let\alpha\pgfmathresult;
    \begin{scope}[rotate=\alpha,color=blue]
      \draw[dashed,->] (0,0)--(0,\delt*\rayon); 
      \draw (0,\delt*\rayon) ellipse ({\z} and {\y*\z});
      \draw (\z,\delt*\rayon); 
      \draw[shift={(\z*\c,\delt*\rayon+\y*\z*\s)},scale=\factor,->] (0,0)--(\s,-\c*\y);
    \end{scope}
    \pgfmathparse{sqrt(\ccc)};  
    \draw (0,\pgfmathresult) node {$\bullet$} node[above] {$y$};
\end{scope}
  \draw (-0.57,0.45) node {{\color{red} $x_1$}};
  \draw (0.28,0.62) node {{\color{blue} $x_2$}};
  \draw (-0.6,0) node {{\color{red} $\partial \Sigma_{x_1}$}};
  \draw (0.7,0.35) node {{\color{blue} $\partial \Sigma_{x_2}$}};
  \draw (2,-0.5) node {$y \in \partial\Sigma_{x_1} \hskip-2pt \cap \partial \Sigma_{x_2}$}; 
  \draw (2,-0.8) node {$T_y \partial \Sigma_{x_1} \neq T_y \partial \Sigma_{x_2}$};  \end{tikzpicture}
\vskip-5pt
\caption{Failure of \eqref{item:SteinCurvature} for spherical Hilbert transforms $H_{{\mathbf{S},\delta}}$} 
\begin{center}
Here $H_{{\mathbf{S},\delta}} = S_\Sigma$ with $\Sigma = \big\{(x,y) \in \mathbf{S}^n \times \mathbf{S}^n : \langle x,y \rangle > \delta \big\}$ for $n = 2$.
\end{center}
\label{picture:sphere}
\end{figure}

Theorem A has profound consequences for Fourier multipliers on Lie group von Neumann algebras. Smooth Fourier multipliers on group algebras were intensively investigated over the last decade \cite{Ca,CGPT2,GJP,JMP1,JMP2,MRX,PRS}. The nonsmooth theory concerns a longstanding search to classify idempotent Fourier $L_p$-multipliers, but the geometric behavior of their symbols is very sensitive to the underlying group. Bo\.zejko and Fendler \cite{BF} studied an analog of Fefferman's ball multiplier theorem in the free group for $|1/p-1/2| > 1/6$, though the general case $p \neq 2$ stands open since then. More recently, Mei and Ricard found a large class of free Hilbert transforms in their remarkable work \cite{MR}. The search for Hilbert transforms on general groups also includes crossed products and groups acting on tree-like structures \cite{GPX,PRo}. 

In this paper, we shall give a complete characterization of the local boundary behavior for completely bounded idempotent Fourier multipliers on arbitrary Lie groups. Our result is more easily stated for simply connected groups. We refer to Section~\ref{sec:fourier} for the statement of the result on general Lie groups (Theorem~\ref{thm:main_Fourier}) and for the precise definitions of local Fourier multipliers.

\begin{TheoB}
\emph{Let $p \in (1,\infty)\setminus\{2\}$. Let $\G$ be a simply connected Lie group, $\A \subset \G$ a $\mathcal{C}^1$-domain and $g_0 \in \partial \Omega$ a point in the boundary of $\A$. The following are equivalent}:
\begin{enumerate}
    \item \emph{$\chi_\A$ defines locally at $g_0$ a completely bounded Fourier $L_p$-multiplier.}

\vskip2pt

    \item \emph{There is a smooth action $\G \to \diff(\R)$ by diffeomorphisms on the real line, such that $\A$ coincides on a neighbourhood of $g_0$ with $\{g \in \G \mid g \cdot 0 > g_0 \cdot 0\}$.}
\end{enumerate}
\end{TheoB}

Alternatively, this means that $\partial \A$ is locally a coset of a codimension $1$ subgroup. 
There are two ingredients in the proof of Theorem B. The first is Theorem A. The second is a general result relating local complete $L_p$-boundedness of Fourier and Schur multipliers for arbitrary locally compact groups. Such a result is known to be true globally at the endpoints $p=1,\infty$ \cite{BF} or when the group $\G$ is amenable \cite{CS,NR}. The eventuality that it could be true locally is a recent observation \cite{PRS} for even integers $p \in 2 \mathbf{Z}_+$ and unimodular groups $\G$. In Theorem~\ref{thm:local_transference_general} below we manage to prove it in full generality as a crucial step towards Theorem B.

Lie himself classified Lie groups admitting (local) actions by diffeomorphisms on the real line \cite{zbMATH02681669}. This classification into three types (translation, affine, and projective) gives rise to the following three fundamental examples of a group $\G$ with a smooth domain $\Omega$:
\begin{itemize}
\item[i)] The real line $\G_1 = \R$ with $\Omega_1 = (0,\infty)$.

\vskip2pt

\renewcommand{\thefootnote}{$\star$}

\item[ii)] The affine group $\G_2 = \Aff_+(\R)$\footnote{Affine increasing bijections $x\mapsto ax+b$ for $a \in \R_+^*$ and $b \in \R$, isomorphic to $\R \rtimes \R_+^*$.} and $\Omega_2 = \{ax+b: b > 0\}$.

\renewcommand{\thefootnote}{$\dag$}

\item[iii)] The universal covering group $\G_3 = \widetilde{\PSL}_2(\R)$\footnote{The action $\alpha: \widetilde{\PSL}_2(\R) \curvearrowright \R$ is obtained by lifting the standard action of $\PSL_2(\R)$ on the projective line to the universal covers. If $p: \R \to P^1(\R)$ denotes the universal cover, then the universal cover of $\SL_2(\R)$ is identified with the group of homeomorphisms $g: \R \to \R$ for which there is $A \in \PSL_2(\R)$ such that $p\circ g= A \cdot p$.} with $\Omega_3 = \{g: \alpha_g(0) > 0\}$.
\end{itemize}
Lie's classification implies the following interesting consequence.

\begin{CorB} \label{CorB1}
\emph{Conditions} (1) \emph{and} (2) \emph{in Theorem} B \emph{are equivalent to}: 
\begin{itemize}
    \item[(3)] \emph{There is $j\in \{1,2,3\}$ and a smooth surjective homomorphism $f\colon \G \to \G_j$ such that the domain $\A$ coincides on a neighbourhood of $g_0$ with $g_0 f^{-1}(\Omega_j)$.}
\end{itemize}
\emph{Therefore, there are three fundamental models for Hilbert transforms on Lie groups.}
\end{CorB}

The above examples define (not only local, but) global completely bounded $L_p$ Fourier multipliers for $1<p<\infty$. Example i) is the classical Hilbert transform and Example iii) follows from recent Cotlar identities for unimodular groups \cite{GPX}. The nonunimodular Example ii) will be properly justified in Example~\ref{ex:cotlar}. 

Theorem B and Corollary B1 give very satisfactory descriptions of completely bounded Fourier idempotents in arbitrary Lie groups. It is certainly surprising that these multipliers are modeled out of exactly three fundamental examples, the classical Hilbert transform and its affine and projective variants. It also shows that every $\mathcal{C}^1$-idempotent is automatically $\mathcal{C}^\infty$. This rigidity property collides head-on with the much more flexible scenario of Theorem A. In the following result, we further describe Fourier $L_p$-idempotents for large classes of Lie groups.

\begin{CorBB}
\emph{Let $p \in (1, \infty) \setminus \{2\}$ and let $\G$ be a Lie group}:
\begin{itemize}
\item[i)] \emph{If $\G$ is simply connected and nilpotent, every cb-$L_p$-bounded smooth Fourier idempotent is locally of the form $H \circ \varphi$, for the Hilbert transform $H$ and some continuous homomorphism $\varphi: \G \to \R$.}

\vskip2pt

\item[ii)] \emph{If $\G$ is a simple Lie group which is not locally isomorphic to $\SL_2(\R)$, then $\G$ does not carry any smooth Fourier idempotent which is locally completely $L_p$-bounded on its group von Neumann algebra.}

\vskip2pt

\item[iii)] \emph{If $\G$ is locally isomorphic to $\SL_2(\R)$, then $\G$ carries a unique local Fourier idempotent which is completely $L_p$-bounded on its group algebra $($up to left/right translations$)$ given by $g \mapsto \frac12 \big( 1 + \mathrm{sgn} \, \mathrm{Tr} (g e_{12}) \big)$.}
\end{itemize}
\end{CorBB}

As an illustration for stratified Lie groups, the homomorphism $\varphi$ corresponds on the Lie algebra level with the projection onto any 1-dimensional subspace of the first stratum. The second statement above spotlights the singular nature of harmonic analysis over simple Lie groups. It also yields an alternative way to answer our motivating question: the spherical Hilbert transform $H_\mathbf{S}$ is not $L_p$-bounded for any $p \neq 2$. Finally, as we shall justify, the third statement gives a straightforward solution (in the negative) to Problem A in \cite{GPX}. We refer to \cite{PisAst,P2} for the operator space background necessary for this paper.

The plan of the paper is as follows. Section~\ref{sec:schur} is devoted to idempotent Schur multipliers. It contains the proof of Theorem~A and several discussions, including our analysis of the spherical Hilbert transform. Section~\ref{sec:fourier} is devoted to Fourier multipliers. It contains the proof of Theorem~B and its corollaries. The proof relies on a result of independent interest on the local transference between  Fourier and Schur multipliers for arbitrary locally compact groups,  Theorem~\ref{thm:local_transference_general}.

\section{\bf \large Idempotent Schur multipliers}\label{sec:schur}

In this section we give a complete proof of Theorem A. We begin by recalling some particularly flexible changes of variables for Schur symbols, which preserve the $S_p$-norm of the corresponding Schur multipliers on nonatomic spaces. Then, we prove the  implications  \eqref{item:Spbounded} $\Rightarrow$ \eqref{item:SteinCurvature} $\Rightarrow$ \eqref{item:HilbertTransform} $\Rightarrow$ \eqref{item:Spbounded} in Theorem A separately. We shall finish with some comments and applications to spherical Hilbert transforms. 

\subsection{Schur multipliers}

Let $(X,\mu)$ and $(Y,\nu)$ be $\sigma$-finite measure spaces. Given $1\leq p<\infty$, let $S_p(L_2(X),L_2(Y))$ be the space all of bounded linear operators $T: L_2(X) \to L_2(Y)$ with $\Tr |T|^p <\infty$, which is a Banach space for the norm below \[\|T\|_{S_p}= \big( \Tr \hskip1pt |T|^p \big)^{\frac 1 p}.\] When $p=2$, the Schatten class $S_2(L_2(X),L_2(Y))$ is the space of Hilbert-Schmidt operators $L_2(X) \to L_2(Y)$. It coincides with $L_2(X\times Y)$, regarding any $L_2$-function $(x,y) \mapsto K(x,y)$ as the kernel of the corresponding Hilbert-Schmidt operator 
\[T_Kf(y) = \int_X K(x,y) f(x) d\mu(x).\] 
Given $m \in L_\infty(X\times Y)$, the Schur $S_p$-multiplier with symbol $m$ is defined (when it exists) as the unique bounded linear map $S_m$ on $S_p(L_2(X),L_2(Y))$ which assigns $T_K = ( K(x,y) )_{x\in X, y \in Y} \in S_2 \cap S_p$ to $( m(x,y) K(x,y) )_{x \in X, y \in Y} = S_m (T_K)$. We shall write $\norm{m}{MS_p}$ for its norm, with the convention  $\norm{m}{MS_p}=\infty$ if $S_m$ does not exist. 

The following general fact will be crucial in our proof of Theorem A. It evidences a much greater flexibility of Schur multipliers compared to Fourier multipliers. The proof follows from \cite{LdlS}, we include the argument below. 

\begin{lem}\label{lem:transfere} 
Let $(X,\mu),(X',\mu'),(Y,\nu),(Y',\nu')$ be atomless $\sigma$-finite measure spaces and $f\colon X \to X'$ and $g\colon Y \to Y'$ be measurable maps. Assume the pushforward measures $f_* \mu$ and $g_* \nu$ are absolutely continuous with respect to the measures $\mu'$ and $\nu'$ respectively. Then, for every $m \in L_\infty(X' \times Y')$
\begin{align*}
\norm{m \circ (f \times g)}{MS_p(L_2(X,\mu),L_2(Y,\nu))} & =  \norm{m}{MS_p(L_2(X',f_*\mu),L_2(Y',g_*\nu))} \\ & \leq \norm{m}{MS_p(L_2(X',\mu'),L_2(Y',\nu'))}.
\end{align*}
\end{lem}

\begin{proof} The last inequality follows directly from \cite[Lemma 1.9]{LdlS} and the absolute continuity assumption. To prove the first identity and lighten the notation, let us assume for simplicity that $(X,\mu)=(Y,\nu)$, $(X',\mu')=(Y',\nu')$ and $f=g$. Let $\mathcal B, \mathcal B'$ be the underlying $\sigma$-algebras, and consider $\mathcal A:=f^{-1}(\mathcal B')$. Then $f$ allows to identify $L_2(X',\mathcal B',f_*\mu)$ with $L_2(X,\mathcal A,\mu)$. In particular we have
\[ \norm{m\circ(f \times f)}{MS_p(L_2(X,\mathcal A,\mu))} = \norm{m}{MS_p(L_2(X',\mathcal B',f_*\mu))},\]
and similarly for the cb-norm. On the other hand, \cite[Lemma 1.13]{LdlS} implies that the cb norms of $m \circ (f\times f)$ on $S_p(L_2(\mathcal A,\mu))$ and $S_p(L_2(\mathcal B,\mu))$ coincide, so we deduce
\[ \norm{m\circ(f \times f)}{\mathrm{cb}MS_p(L_2(X,\mathcal B,\mu))} = \norm{m}{\mathrm{cb}MS_p(L_2(X',\mathcal B',f_*\mu))}.\]
Then \cite[Theorem 1.18]{LdlS} allows us to conclude. Indeed, our assumptions that $\mu$ and $\mu'$ have no atoms imply that both cb-norms are equal to their norms. 
\end{proof}   

\subsection{Proof of Theorem A: Boundedness implies zero-curvature} 
In this paragraph we prove \eqref{item:Spbounded} $\Rightarrow$ \eqref{item:SteinCurvature} from the statement of Theorem A. The key new idea we introduce is an amplification of Meyer's lemma \cite[Lemma 1]{Fe}, which directly connects Schur idempotents with Meyer's classical condition, to which Fefferman's construction may be applied. Taking charts, we can and will assume that $M$ and $N$ are open subsets of $\R^{m}$ and $\R^n$ respectively. We shall further assume for simplicity that $m=n$, our argument applies as well when $m \neq n$. Given $z \in \partial \Sigma$, let $\mathbf{n}(z) = (\mathbf{n}_1(z),\mathbf{n}_2(z)) \in \R^n \oplus \R^n$ be a normal to $\partial \Sigma$ at $z$ pointing towards $\Sigma$. 

\begin{lem}\label{prop:Meyer_Lemma_NC}
Consider a $\mathcal{C}^1$-domain $\Sigma \subset \R^n \times \R^n$ and open sets $U,V \subset \R^n$ such that $\partial \Sigma$ intersects $U \times V$. Let $S_\Sigma$ be the idempotent Schur multiplier whose symbol is the characteristic function of $\Sigma$ and assume that it is bounded on $S_p(L_2(U),L_2(V))$ with norm $C$. Let $x_1, x_2, \ldots, x_N \in U$ and $y \in V$ such that $z_j = (x_j,y)$ is a transverse point in the boundary $\partial \Sigma$ for every $j=1,2,\ldots,N$. Define $u_j = \mathbf{n}_2(z_j)$ and consider functions $f_1, f_2, \ldots, f_N \in L_p(\R^n)$. Then, we have 
\[ \Big\| \Big( \sum_{j=1}^N \big| H_{u_j}(f_j) \big|^2 \Big)^{\frac{1}{2}} \Big\|_{L_p(\R^n)} \leq C \Big\| \Big( \sum_{j=1}^N |f_j|^2 \Big)^{\frac{1}{2}} \Big\|_{L_p(\R^n)}. \]
Here we write $H_u$ for the $u$-directional Hilbert transform $\widehat{H_u f}(\xi) = \chi_{\langle \xi,u\rangle >0} \widehat{f}(\xi)$.
\end{lem}

The implication \eqref{item:Spbounded}$\Rightarrow$\eqref{item:SteinCurvature} in Theorem A follows from Lemma \ref{prop:Meyer_Lemma_NC}. Indeed, by the transversality assumption, the map $z \mapsto \mathbf{n}_2(z)/\|\mathbf{n}_2(z)\|$ is continuous on a neighbourhood of the transverse point $(x_0,y_0)$ in Theorem A. Moreover, for $y$ close to $y_0$ we have that $\partial \Sigma^y$ is locally a manifold, so is connected. Therefore, if \eqref{item:SteinCurvature} was not true, there would exist $y$ close to $y_0$ such that the subset of the sphere $X = \{\mathbf{n}_2(x',y)/\|\mathbf{n}_2(x',y)\|: x' \in \partial \Sigma^y \cap U\}$ contains a connected subset not reduced to a point. According to \eqref{item:Spbounded} and Lemma \ref{prop:Meyer_Lemma_NC}, this would imply that the square function inequality there holds uniformly in $L_p(\R^n)$ for any finite set in a continuum of directions in the $(n-1)$-sphere. However, Fefferman's main result in his proof of the ball multiplier theorem \cite{Fe} claims that such a uniform inequality cannot hold. In fact, Fefferman stated it for $n=2$ but the result in arbitrary dimension follows from the $2$-dimensional case by K. de Leeuw's restriction theorem \cite{dL}. In particular, the zero-curvature condition \eqref{item:SteinCurvature} must hold. 

\begin{proof}[Proof of Lemma \emph{\ref{prop:Meyer_Lemma_NC}}.] The proof relies on the following two claims:
\begin{itemize}
\item[(A)] Let $(x,y)$ be a transverse point in the boundary of $\partial \Sigma$ and let $T \in \GL_n(\R)$ be such that $T^* \mathbf{n}_1(x,y) = - \mathbf{n}_2(x,y)$. Then, the following identity holds for almost every $\xi,\eta \in \R^n$ 
\[\hskip20pt \lim_{\varepsilon\to 0^+} \chi_\Sigma \big( x+\varepsilon T\xi, y+\varepsilon \eta \big) = \frac12 \big( 1 + \mathrm{sgn} \big\langle \mathbf{n}_2(x,y) , \eta - \xi \big\rangle \big).\]

\item[(B)] Let $u_j$ be as in the statement. Then the Schur multiplier  
\[\hskip17pt \big( (\xi,j),\eta \big) \in \big( \R^n \times\{1,\dots,N\} \big) \times \R^n \mapsto \frac12 \big( 1+ \mathrm{sgn} \langle \eta - \xi,u_j \rangle \big)\] is bounded on $S_p(L_2(\R^n),L_2(\R^n \times \{1,2,\ldots,N\}))$ with norm $\le \norm{\chi_\Sigma}{MS_p}$.
\end{itemize}
Assuming the validity of the above claims, we may now conclude the proof using standard transference ideas that go back at least to the work of Bo{\.z}ejko and Fendler \cite{BF0}. We know from \cite[Theorem 5.2]{CS} that there is an ultrafilter $\cU$ on $\N$ and a completely isometric map \[j_p\colon L_p(\R^n) \to \prod_\cU S_p(L_2(\R^n))\] that intertwines Fourier and Schur multipliers. This gives $j_p(H_u(f)) = S_{m_u}(j_p(f))$ for every $u \in \R^n$ and $f \in L_p(\R^n)$. Here $m_u(\xi,\eta) = \frac12 (1 + \mathrm{sgn} \langle \eta-\xi,u \rangle )$. As a consequence we have for every $f_1,f_2,\ldots,f_N \in L_p(\R^n)$
\begin{eqnarray*} 
\Big\| \Big( \sum_{j=1}^N |H_{u_j}(f_j)|^2 \Big)^{\frac{1}{2}} \Big\|_{L_p(\R^n)} \!\!\! & = & \!\!\! \Big\| \sum_{j=1}^N H_{u_j}(f_j) \otimes e_{j,1} \Big\|_{L_p(\R^n;S_p)} \\ \!\!\! & = & \!\!\! \Big\| \sum_{j=1}^N j_p(H_{u_j}(f_j)) \otimes e_{j,1} \Big\|_{S_p} \\ \!\!\! & = & \!\!\! \Big\| \sum_{j=1}^N S_{m_{u_j}} (j_p(f_j)) \otimes e_{j,1} \Big\|_{S_p}.
\end{eqnarray*}
According to claim (B) we deduce 
\begin{eqnarray*}
\Big\| \Big( \sum_{j=1}^N |H_{u_j}(f_j)|^2 \Big)^{\frac{1}{2}} \Big\|_{L_p(\R^n)} \!\!\! & \leq & \!\!\! \|\chi_\Sigma\|_{MS_p} \Big\| \sum_{j=1}^N j_p(f_j) \otimes e_{j,1} \Big\|_{S_p} \\ \!\!\! & = & \!\!\! \|\chi_\Sigma\|_{MS_p} \Big\| \Big( \sum_{j=1}^N |f_j|^2 \Big)^{\frac{1}{2}} \Big\|_{L_p(\R^n)}.
\end{eqnarray*}

Thus, the assertion is a consequence of claim (B), for which we need to justify claim (A) first. To do so, we can assume that $\Sigma = f^{-1}(0,\infty)$ for a $\mathcal{C}^1$-submersion $f \colon U \times V \to \R$. Then $\nabla f(x,y)= (\nabla_x f(x,y),\nabla_y f(x,y))$ is a normal vector to the boundary $\partial \Sigma = f^{-1}(0)$ at every $(x,y) \in \partial \Sigma$. Thus, replacing $f$ by a multiple, we can assume that its gradient is $(\mathbf{n}_1(x,y),\mathbf{n}_2(x,y))$. Then, the Taylor expansion of $f$ gives
\begin{eqnarray*} 
f \big( x+\varepsilon T \xi, y+\varepsilon \eta \big) & = & \varepsilon \big\langle \mathbf{n}_1(x,y), T\xi \big\rangle + \varepsilon \big\langle \mathbf{n}_2(x,y), \eta \big\rangle + o(\varepsilon) \\ & = & \varepsilon \big\langle \mathbf{n}_2(x,y), \eta - \xi \big\rangle + o(\varepsilon).
\end{eqnarray*}
Therefore, if $\eta - \xi$ is not orthogonal to $\mathbf{n}_2(x,y)$ (a condition that holds for almost every $\xi$ and $\eta$), we have $\chi_\Sigma (x+\varepsilon T\xi, y+\varepsilon \eta) = \frac12 (1 + \mathrm{sgn} \langle \mathbf{n}_2(x,y) , \eta - \xi \rangle)$ for every $\varepsilon>0$ small enough. This proves claim (A).
 
To prove claim (B) we apply (A). More precisely, let $T_j \in \GL_n(\R)$ be such that $T_j^* \mathbf{n}_1(x_j,y) = - \mathbf{n}_2(x_j,y) = - u_j$ for every $j = 1,2,\ldots,N$. The existence of these maps is clear, because by the transversality assumption both $\mathbf{n}_1(x_j,y)$ and $\mathbf{n}_2(x_j,y)$ are nonzero vectors in $\R^n$ and $\GL_n(\R)$ acts transitively on them. By  Lemma~\ref{lem:transfere} the Schur multiplier with symbol 
\[ m_\varepsilon \big( (\xi,j),\eta \big) = \chi_\Sigma \big( x_j + \varepsilon T_j \xi, y + \varepsilon \eta \big) \]
is bounded with norm $\leq \norm{\chi_\Sigma}{MS_p}$ for every $\varepsilon>0$. Taking $\varepsilon \to
0^+$, we obtain that the almost everywhere limit of $m_\varepsilon$ is $S_p$-bounded with norm $\leq\norm{\chi_\Sigma}{MS_p}$. 
However this limit is $( 1+ \mathrm{sgn} \langle \eta - \xi,u_j \rangle)/2$ from claim (A). This proves claim (B). \end{proof}

\begin{rem}
\emph{Taking \[\Sigma = \big\{ (x,y): x-y \in \Omega \big\}\] for some smooth domain $\Omega$, Lemma \ref{prop:Meyer_Lemma_NC} reduces to the classical Meyer's lemma.}
\end{rem}

\subsection{Proof of Theorem A: Zero-curvature implies triangular truncations}

The implication \eqref{item:SteinCurvature}$\Rightarrow$\eqref{item:HilbertTransform} is a general geometric statement concerning transverse hypersurfaces in manifolds of product type. Let $M,N$ be manifolds of dimension $m,n$. Following the terminology in the Introduction, we say that a $\mathcal{C}^1$-submanifold $\manifold \subset M \times N$ of codimension $1$ is said to be transverse at $z=(x,y) \in \manifold$ if the tangent space of $\manifold$ at $z$ maps surjectively on each factor $T_x M$ and $T_y N$. In that case, $\manifold_x = \{y' \in N \mid (x,y') \in \manifold\}$ and $\manifold^y = \{x' \in M \mid (x',y) \in \manifold\}$ are manifolds on a neighbourhood $y$ and $x$ respectively.
\begin{thm}\label{thm:hypersurfaces_in_products}
Let $\manifold \subset M \times N$ be a $\mathcal{C}^1$-submanifold of codimension $1$ that is transverse at $z_0=(x_0,y_0)\in \manifold$. Then, the following are equivalent$\hskip1pt :$
  \begin{enumerate}
      \item\label{item:SteinCurvature1} There are neighbourhoods $U,V$ of $x_0$ and $y_0$ in $M,N$ such that for every $x,x' \in U$ and $y \in V$ with $(x,y),(x',y) \in \manifold$, $T_y \manifold_x = T_y \manifold_{x'}$.
      \item\label{item:SteinCurvature2} There are neighbourhoods $U,V$ of $x_0$ and $y_0$ in $M,N$ such that for every $x \in U$ and $y,y' \in V$ with $(x,y),(x,y') \in \manifold$, $T_x \manifold^y = T_x \manifold^{y'}$.
      \item\label{item:HilbertTransform1} There are neighbourhoods $U,V$ of $x_0$ and $y_0$ in $M,N$ and $C^1$ submersions $f \colon U \to \R$ and $g \colon V \to \R$ with $\manifold \cap (U \times V) = \{(x,y) \in U \times V \mid f(x)=g(y)\}$.
  \end{enumerate}
\end{thm}
By the symmetry of the two variables, it is enough to prove the equivalence \eqref{item:SteinCurvature1}$\Leftrightarrow$\eqref{item:HilbertTransform1}. The implication \eqref{item:HilbertTransform1}$\Rightarrow$\eqref{item:SteinCurvature1} is clear with the same $U$ and $V$, because in that case $T_y \manifold_x$ is the kernel of $d_y g$, which is independent of $x$. The converse is less direct. Both conditions are invariant by diffeomorphisms of product type, that is of the form $(x,y)\mapsto (\phi(x),\psi(y))$. It will be useful to have a description of a local normal form (that is of an element in every orbit) of transverse manifolds.
\begin{lem}\label{lem:normal_form}
Consider a $\mathcal{C}^1$-submanifold $\manifold \subset M \times N$ of codimension $1$ that is transverse at $z_0=(x_0,y_0)\in \manifold$. Then, there are diffeomorphisms $\phi$ and $\psi$ from neighbourhoods $U$ and $V$ of $x_0$ and $y_0$ respectively into $\R^m$ and $\R^n$ satisfying that $\phi(x_0) = 0 = \psi(y_0)$ and such that $$\manifold \cap (U\times V) = (\phi\times \psi)^{-1} \big\{ (x,y) \mid x_1 =  g(x_2,\dots,x_m,y)\big\}$$ for some $\mathcal{C}^1$ function $g \colon \R^{m-1} \times \R^n \to \R$ satisfying $g(0,y)=y_1$ for every $y$.
\end{lem}

\begin{proof}
By the transversality assumption that $T_{z_0}\manifold$ surjects onto $T_{y_0} M$, we see that $T_{z_0} \manifold \cap (T_{x_0}M \oplus 0) \neq T_{x_0} M \oplus 0$. Thus, by applying a local diffeomorphism $\phi:M \to \R^m$ we can assume that $M=\R^m$, $x_0=0$ and $(1,0, \ldots,0)\notin T_{z_0} \manifold$. Then by the implicit function theorem, there is a $\mathcal{C}^1$ function $h \colon \R^{m-1} \times N \to \R$ such that, on a neighbourhood of $(0,y_0)$ $$\manifold = \big\{(x,y) \mid x_1 = h(x_2,\dots,x_m,y)\big\}.$$ The function $h(0,\cdot)$ vanishes at $y_0$ and, by the second half of the transversality assumption, has nonzero differential at $y_0$. By the implicit function theorem (or the surjection theorem) again, there is a diffeomorphism $\psi$ from a neighbourhood of $y_0$ into $\R^n$ vanishing at $y_0$ and such that $h(0,y) = \psi(y)_1$ for every $y$ close enough to $y_0$. This proves the lemma with $g(0,y)=h(0,\psi^{-1}(y))$.
\end{proof}


Now we can prove the implication \eqref{item:SteinCurvature1}$\Rightarrow$\eqref{item:HilbertTransform1} in Theorem \ref{thm:hypersurfaces_in_products}. Observe that both conditions are unchanged if we replace $(x_0,y_0,\manifold)$ by $(\phi(x_0),\psi(y_0),\phi \times \psi(\manifold))$ for local diffeomorphisms. Therefore, by the normal form Lemma~\ref{lem:normal_form}, we may assume that $M \times N=\R^m \times \R^n$, $(x_0,y_0)=(0,0)$ and $$\Pi \cap (U\times V) = \big\{ (x,y)\in U \times V \mid x_1 = g(x_2,\dots,x_m,y)\big\}$$ for some $\mathcal{C}^1$ function $g\colon \R^{m-1}\times\R^n \to \R$ satisfying $g(0,y)=y_1$. Then, for every $(x,y) \in \manifold$ with $x=(x_1,\tilde{x})$, we have $T_y \manifold_x = \ker( d_y g(\tilde x,y))$. Let $\tilde U \subset \R^{m-1},\tilde V \subset V$ be square neighbourhoods of $0$ such that $(g(\tilde x,y),\tilde x) \in U$ for $(\tilde x,y) \in \tilde U \times \tilde V$. Then for every such $\tilde x,y$, condition \eqref{item:SteinCurvature1} applied to $x=(g(\tilde x,y),\tilde x)$ and $x' = (g(0,y),0)$ yield $\ker d_y g( \tilde x,y) = \ker d_y g(0,y) = \vect(e_2,\dots,e_n)$. In particular, $\partial_{y_j} g(\tilde x,y)=0$ for every $j \geq 2$, so (since $\tilde V$ is a square) $g(\tilde x,y) = w(\tilde x,y_1)$ for certain $\mathcal{C}^1$ function $w:\R^{m-1}\times \R \to \R$ satisfying $w(0,s)=s$ for all $s$. By the implicit function theorem, we get $$\big\{(x_1,\tilde x,s) \mid x_1= w(\tilde x,s) \big\} = \big\{(x_1,\tilde x,s) \mid s = u(x_1,\tilde x) \big\}$$ locally for a $\mathcal{C}^1$ function $u \colon \R^m \to \R$. This completes the proof of Theorem \ref{thm:hypersurfaces_in_products}.

\subsection{Proof of Theorem A: Transference on triangular truncations}\label{sec:transference}

The implication \eqref{item:HilbertTransform} $\Rightarrow$ \eqref{item:Spbounded} in Theorem A is immediate from the boundedness of the triangular projection on Schatten $p$-classes for $1 < p < \infty$ and the transference Lemma \ref{lem:transfere} above. This completes the proof of Theorem A. \fin

\subsection{Relatively compact domains.} Using a partition of unity argument, it is not difficult to prove that Theorem A holds globally for relatively compact fully transverse domains $\Sigma$. More precisely, let $p \in (1,\infty)\setminus\{2\}$ and consider a relatively compact domain $\Sigma$ in $M \times N$ which is transverse at every point of $\partial \Sigma$. Then $S_\Sigma$ is an $S_p$-bounded multiplier if and only if any of the equivalent conditions \eqref{item:SteinCurvature} and \eqref{item:HilbertTransform} in the statement of Theorem A holds at every point of the boundary. 

 \begin{rem}\label{rem:local_vs_global}
\emph{The fact that $\Sigma$ is relatively compact is crucial in the preceding argument. For instance,  \eqref{item:SteinCurvature} holds trivially at every boundary point for every fully transverse $\mathcal{C}^1$-domain of $\R\times \R$. But there are examples of such domains |which are Toeplitz, arising from Fourier symbols| that do not define an $S_p$ multiplier for any $p \neq 2$. An explicit construction is given in \cite[Appendix A]{CPPR}.}
 \end{rem}

At this point, it is interesting to observe the difference here between Fourier and Schur idempotents. We know from Fefferman's theorem \cite{Fe} that there are no Fourier $L_p$-idempotents associated to smooth compact domains. However, there are plenty such Schur idempotents: necessarily nonToeplitz, since Toeplitz symbols give rise to Fourier idempotents. A funny instance is precisely given by other forms of ball multipliers $\Sigma_R = \{ (x,y) \hskip-1pt \in \hskip-1pt \R^n \times \R^n\hskip-2pt : |x|^2 + |y|^2 < R^2 \}$, which are clearly $S_p$-bounded and have been recently used by Chuah-Liu-Mei in their recent paper \cite[Example 4.4]{CLM}. Theorem A proves in addition that the spheres $\partial \Sigma_R$ satisfy the zero-curvature condition \eqref{item:SteinCurvature}. More intriguing examples are the spherical Hilbert transforms defined in the Introduction as 
\[ H_\mathbf{S}: A \mapsto \Big( \mathrm{sgn}\langle x,y \rangle A_{xy} \Big)_{x,y \in \mathbf{S}^n}. \] More generally, we also define $H_{\mathbf{S},\delta}(A) = (\chi_{\langle x,y \rangle > \delta} A_{xy})$ for any $\delta \in (-1,1)$. The case $\delta = 0$ corresponds to $\frac12 (1 + \mathrm{sgn} \langle x,y \rangle)$ which is formally equivalent to $H_\mathbf{S}$ above. 

\begin{cor} \label{Cor-SHT}
Let us fix $1 < p \neq 2 < \infty$. Then, the $n$-dimensional spherical Hilbert transforms $H_{\mathbf{S},\delta}$ are all $S_p$-bounded for $n = 1$ and $S_p$-unbounded for $n \ge 2$. 
\end{cor}

\begin{proof} Spherical Hilbert transforms arise from relatively compact domains whose boundary is fully transverse. In particular, we may apply Theorem A. Next, in dimension $1$ the assertion follows since the zero-curvature condition \eqref{item:SteinCurvature} is trivially satisfied. Alternatively, the symbol can be expressed as a triangular truncation in terms of the polar coordinates of $x$ and $y$. When $n \ge 2$, it is easily checked that the tangent spaces at $\partial \Sigma_{x_1}$ and $\partial \Sigma_{x_2}$ differ at their intersection points. This was illustrated for $n=2$ in Figure~\ref{picture:sphere}. Theorem A implies the assertion. \end{proof}

\begin{rem}
\emph{Alternatively, Corollary~\ref{Cor-SHT} also follows as a special case of Corollary B2. Indeed, Lemma~\ref{lem:transfere} implies that $H_{\mathbf{S},\delta}$ has the same norm as the Schur multiplier on $\SO(n+1) \times \SO(n+1)$ with symbol $(g,h)\mapsto \mathrm{sgn}( (g^{-1}h)_{1,1})$, which by \cite{CS} coincides with the cb-norm of the Fourier multipler with symbol $g\mapsto \mathrm{sgn}(g_{1,1})$. But for $n\geq 2$ $\SO(n+1)$ is a simple Lie group not locally isomorphic to $\SL_2(\R)$, so it does not carry \emph{any} idempotent multiplier.}
\end{rem}
\begin{rem}
\emph{We may also consider symbols $\Sigma_\delta = \{(x,y) \in \R^n: \langle x,y \rangle > \delta \}$ in the full Euclidean space for $n \ge 2$. In this case, Theorem A gives $S_p$-unboundedness for $(n,\delta) \neq (2,0)$. By \cite[Theorem 1.18]{LdlS} and since $S_{\Sigma_0} = H_{\mathbf{S},0} \otimes \mathrm{id}_{\R_+}$, it turns out that $S_p$-boundedness for $(n,\delta) = (2,0)$ follows from Corollary~\ref{Cor-SHT}.}
\end{rem}

\subsection{Curvature on smoother domains.} Our curvature condition \eqref{item:SteinCurvature} admits an alternative formulation under additional regularity. Let $\Sigma$ be a $\mathcal{C}^2$-domain. Then $\Sigma \cap (U \times V) = \big\{ (x, y) : F(x, y) > 0 \big\}$ for some $\mathcal{C}^2$-function $F: M \times N \to \R$ and small enough neighbourhoods $U,V$. Our curvature condition holds if and only if we have 
\[ \Big\langle d_x \hskip-1pt d_yF(x, y), u \otimes v \Big\rangle := u^\mathrm{t} \cdot \Big( \partial_{x_j} \partial_{y_k} F(x,y) \Big)_{j,k} \cdot v = 0 \]
for $(u,v) \in \ker d_xF(x, y) \times \ker d_yF(x, y)$ at every $(x, y) \in \partial \Sigma \cap (U \times V)$. The argument is quite simple. By fixing boundary points $(x,y)$ and vectors $(u,v)$ as specified above, let $\gamma: [0,1] \to \partial \Sigma^y \cap U$ be a curve with $\gamma(0) = x$ and $\gamma'(0)=u$, and set $h(s) = d_yF(\gamma(s),y)$. The curvature condition \eqref{item:SteinCurvature} means that $h(s) = \alpha(s) h(0)$ for some nonvanishing function $\alpha: [0,1] \to \R$. In particular, we get 
\[\Big\langle d_x \hskip-1pt d_yF(x, y), u \otimes v \Big\rangle = \big\langle h'(0),v \big\rangle = \alpha'(0) \big\langle h(0),v \big\rangle = 0.\]
Reciprocally, assume that the $\mathcal{C}^2$-curvature condition above holds. Consider a curve $\gamma\colon [0,1] \to \partial \Sigma^y \cap U$ and define $h$ as above. Since we have $\gamma'(s) \in \ker d_xF(\gamma(s),y)$ and $h'(s) = \gamma'(s)^\mathrm{t} \cdot d_xd_yF(\gamma(s),y)$ by construction, it turns out that $\langle h'(s),v \rangle$ equals $\langle d_xd_yF(\gamma(s),y), \gamma'(s) \otimes v \rangle$ for any $v \in \ker d_yF(\gamma(s),y)$. Applying the $\mathcal{C}^2$-curvature condition, this implies that $h'(s)$ is parallel to $h(s)$ for every $s$, which leads to the ODE 
\[\left. \begin{array}{rcl} h'(s) & = & \lambda(s) h(s) \\ h(0) & = & d_yF(x,y) \end{array} \right\} \Rightarrow h(s) = \exp \Big( \int_0^s \lambda(t)dt \Big) h(0) = \alpha(s) h(0)\] for a nonvanishing $\alpha: [0,1] \to \R$. This implies condition \eqref{item:SteinCurvature} in Theorem A.

\begin{rem}
\emph{In this form, \eqref{item:SteinCurvature} is invariant under exchanging $x$ and $y$, which is clear a posteriori without the $\mathcal{C}^2$ assumption, since both \eqref{item:Spbounded} and \eqref{item:HilbertTransform} are. On the other hand, condition \eqref{item:SteinCurvature} in Theorem A seems new, while its $\mathcal{C}^2$-form above is quite similar to the \emph{rotational curvature} $\det [d_xd_y F(x,y)] $ defined by Stein in \cite[XI.3.1]{St}}.
\end{rem}

\subsection{On the transversality condition.} 
The transversality assumption has been essential in our proofs of \eqref{item:Spbounded}$\Rightarrow$\eqref{item:SteinCurvature}$\Rightarrow$\eqref{item:HilbertTransform} in Theorem A, but it is not clear that it is really needed for the statement. Indeed, conditions \eqref{item:Spbounded} and \eqref{item:HilbertTransform} make sense without it, and \eqref{item:SteinCurvature} is already meaningful if one only assumes that $\mathbf{n}_2(x_0,y_0) \neq 0$, and we do not have an example where the equivalence fails. It is likely that such examples can be found, but probably not for domains with analytic boundary. We leave these questions as open problems. In the degenerate case where $\mathbf{n}_1$ is identically $0$, or equivalently when $\Sigma$ is locally of the form $\Sigma = \{(x,y): y \in \Omega \}$, all conditions in Theorem~A hold. The $S_p$-boundedness is in that case even true for $1 \leq p \leq \infty$ because the Schur multiplier whose symbol is the indicator function of $\Sigma$ is just the right-multiplication by the orthogonal projection on $L_2(\Omega)$. 

\section{\bf \large Idempotent Fourier multipliers on Lie groups}\label{sec:fourier}
Let $\G$ be a Lie group, that we equip with a left Haar measure. As to every locally compact group, we can associate to it: 
\begin{itemize}
    \item Its von Neumann algebra $\cL \G$. 
    \item The noncommutative $L_p$ spaces $L_p(\cL \G)$ for $1 \leq p < \infty$.
    \item The Fourier $L_p$-multipliers $T_m$ with symbol $m\colon \G \to \C$.
\end{itemize}
The group von Neumann algebra $\cL \G$ is the weak-$*$ closure in $B(L_2(\G))$ of the algebra of convolution operators $\lambda(f):\xi \in L_2(\G) \mapsto f \ast \xi$ for $f \in \mathcal{C}_c(\G)$. When $\G$ is unimodular, its $L_p$-theory is quite elementary: $\cL \G$ carries a natural semifinite trace $\tau$ given by $\tau( \lambda(f)^* \lambda(f)) = \int |f(g)|^2 dg$ for every $f \in L_2(\G)$ with $\lambda(f) \in \cL \G$; $L_p(\cL \G)$ is then defined as the completion of $\{x \in \cL \G : \|x\|_p<\infty\}$ for the norm $\|x\|_p = \tau( |x|^p)^{1/p}$. It turns out that $L_p(\cL \G)$ contains $\{\lambda(f) : f \in \mathcal{C}_c(\G)\ast \mathcal{C}_c(\G)\}$ as a dense subspace. A bounded measurable $m \colon \G \to \C$ defines a Fourier $L_p$-multiplier if $\lambda(f) \mapsto \lambda(mf)$ extends to a bounded map $T_m$ on $L_p(\cL \G)$. These definitions are more involved for nonunimodular groups and will be recalled in Section~\ref{sec:nonunimodular} below. 

When $p=1,\infty$, a bounded measurable function $m \colon \G\to\C$ defines a completely bounded Fourier $L_p$-multiplier if and only if the Schur multiplier associated to the symbol $(g,h)\mapsto m(g h^{-1})$ |called the Herz-Schur multiplier with symbol $m$ and denoted $S_m$| is completely $S_p$-bounded, with same norms \cite{BF}. For amenable groups, the same holds for $1 < p < \infty$ \cite{CS,NR}, and it is an intriguing open problem whether this holds beyond amenable groups. We shall use that this always holds locally. This phenomenon was discovered recently \cite[Theorem 1.4]{PRS} when $p$ is an even integer and $\G$ unimodular, and the following generalizes this to the general case, see \cite{CJKM} for other local results of similar nature. In what follows, the Fourier support of an element $x \in L_p(\cL \G)$ will refer to the smallest closed subset $\Lambda$ such that $T_m(x)=0$ for every Fourier $L_p$-multiplier with symbol $m$ whose support is a compact subset of $\G\setminus \Lambda$. When $\G$ is unimodular and $x=\lambda(f)$ for $f \in \mathcal{C}_c(\G) \ast \mathcal{C}_c(\G)$ it is easy to see that this coincides with the support of the function $f$.

\begin{thm}\label{thm:local_transference_general} Let $\G$ be a locally compact group and consider a bounded measurable function $m \colon \G \to \C$. Then, the following are equivalent for $p \in (1,\infty)$ and $g_0 \in \G$ $\hskip-3pt :$
  \begin{enumerate}
  \item\label{item:localFourier} There is a neighbourhood $U$ of $g_0$ such that the restriction $T_{m,U}$ of $T_m$ to the space of elements of $L_p(\cL \G)$ Fourier supported in $U$ is completely bounded.

\vskip2pt

  \item\label{item:extensionFourier} There exists a function $\varphi \colon \G \to \C$ which equals $1$ on a neighbourhood of $g_0$ such that $\varphi m$ defines a completely bounded Fourier multiplier on $L_p(\cL \G)$.

\vskip2pt

  \item\label{item:localHerzSchur} There are open sets $V,W \subset \G$ with $g_0 \in V W^{-1}$ such that the function $(g,h) \in V \times W \mapsto m(g h^{-1})$ defines a completely bounded Schur multiplier on the Schatten class $S_p(L_2(V),L_2(W))$. 
  \end{enumerate}
\end{thm}

When these conditions hold, we say that $m$ defines locally at $g_0$ a completely bounded Fourier $L_p$-multiplier. The proof is given in Section~\ref{sec:proof_transference}. We can record the following consequence, which is immediate by looking at condition \eqref{item:localHerzSchur}.
\begin{cor}\label{cor:lifting_multipliers}
Let $\G$ be a connected Lie group and denote by $\widetilde \G$ its universal cover. Let $\tilde{g}_0 \in \widetilde{\G}$ be any lift of $g_0 \in \G$. Then $m \colon \G \to \C$ defines locally at $g_0$ a completely bounded Fourier $L_p$-multiplier over $\cL \G$ if and only its lift $\widetilde{m}$ defines locally at $\tilde{g}_0$ a completely bounded Fourier $L_p$-multiplier as well.
\end{cor}

\subsection{Idempotent multipliers}
Now we are ready to prove Theorem B and also Corollaries B1 and B2 from the Introduction. In fact, we shall prove a slightly expanded version of Theorem~B which includes non simply connected groups and Corollary~B1 at once. The groups $\G_1,\G_2,\G_3$ in the statement below refer to the real line $\R$, $\Aff_+(\R)$ and $\widetilde{\PSL}_2(\R)$ as in the Introduction.

\begin{thm}\label{thm:main_Fourier}
Let $p \in (1,\infty)\setminus\{2\}$. Let $\G$ be a connected Lie group, $\A \subset \G$ a $\mathcal{C}^1$-domain and $g_0 \in \partial \Omega$ a point in the boundary of $\A$. Consider the following conditions$\hskip1pt :$
\begin{enumerate}
    \item \label{item:local-Hilbert-tranform} $\chi_\A$ defines locally at $g_0$ a completely bounded Fourier $L_p$-multiplier.

\vskip5pt

    \item \label{item:local_subgroup} There is a smooth action $\G \to \diff(\R)$ by diffeomorphisms on the real line, such that $\A$ coincides on a neighbourhood of $g_0$ with $\{g \in \G \mid g \cdot 0 > g_0 \cdot 0\}$.

\vskip3pt

    \item \label{item:three-HTs} There is $j\in \{1,2,3\}$ and a smooth surjective homomorphism $f\colon \G \to \G_j$ such that the domain $\A$ coincides on a neighbourhood of $g_0$ with $g_0 f^{-1}(\Omega_j)$. 

\vskip3pt

    \item \label{item:Lie-subalg} $\partial \A = g_0 \exp(\mathfrak{h})$ locally near $g_0$ for some codimension $1$ Lie subalgebra $\mathfrak{h} \subset \mathfrak{g}$.
\end{enumerate}
Then $(1) \Leftrightarrow (4) \Leftarrow (2) \Leftrightarrow(3)$. If $\G$ is simply connected, then we also have $(4) \Rightarrow (2)$.
\end{thm}

\begin{proof} The main difficulty is to prove the equivalence $\eqref{item:local-Hilbert-tranform}\Leftrightarrow\eqref{item:Lie-subalg}$, which we leave to the end of the proof. The implication $\eqref{item:local_subgroup}\Rightarrow\eqref{item:Lie-subalg}$ is clear, with $\mathrm{H} = \exp(\mathfrak{h})$ the stabilizer of $0$. Under the assumption that $\G$ is simply connected, the converse $\eqref{item:Lie-subalg}\Rightarrow\eqref{item:local_subgroup}$ holds by a Theorem of Mostow \cite{Mo}, which implies that $\mathrm{H}$ is a closed subgroup. Therefore, $\G/\mathrm{H}$ is a $1$-dimensional manifold that is simply connected, so is diffeomorphic to $\R$. The implication $\eqref{item:three-HTs}\Rightarrow\eqref{item:local_subgroup}$ is also clear because $\G_j$ is given as a group of diffeomorphisms of $\R$ with $\Omega_j=\{g \in \G_j\mid g \cdot 0>0\}$. The converse $\eqref{item:local_subgroup}\Rightarrow\eqref{item:three-HTs}$ follows from Lie's classification of (local) actions by diffeomorphism on the real line \cite{zbMATH02681669}, see also \cite{MR0120308} for modern presentations and \cite{Gh} for the global aspect. More precisely, the fact that $g_0$ belongs to the boundary of $\A$ implies that $0$ is not fixed by $\G$ and |here we use that $\G$ is connected| the $\G$-orbit of $0$ is an open interval, so by identifying it with $\R$ we can assume that the $\G$-action is transitive. In that case, the image of $\G$ in $\diff(\R)$ is one of the three groups in condition \eqref{item:three-HTs} of Corollary B1, see \cite[Section 4.1]{Gh} for the details.

Next, let us focus on the equivalence $\eqref{item:local-Hilbert-tranform}\Leftrightarrow\eqref{item:Lie-subalg}$ for general Lie groups. If we translate $\A$ by $g_0^{-1}$, we may assume that $g_0=e$ and the tangent space of $\G$ at $g_0$ identifies with its Lie algebra $\mathfrak{g}$. Also, the tangent space of $\partial \A$ at $g_0$ identifies with a codimension $1$ subspace $\mathfrak h$ of $\mathfrak g$. Define the $\mathcal{C}^1$-manifold
\[\widetilde \A = \Big\{(g,h) \in \G\times \G \mid gh \in \Omega\Big\}.\]
Its sections $\widetilde{\A}_g$ and $\widetilde{\A}^h$ are left and right translates of the $\mathcal{C}^1$-domain $\A$
\begin{equation}\label{eq:identification_sheets_Omega}
    \widetilde{\A}_g:= \big\{h: (g,h) \in \widetilde{\A}\big\} = g^{-1}\A \quad \mbox{and} \quad \widetilde{\A}^h := \big\{g : (g,h) \in \widetilde
\A\big\}=\A h^{-1}.
\end{equation}  In particular, $\widetilde \A$ is transverse at every point of its boundary. By Lemma \ref{lem:transfere} and Theorem~\ref{thm:local_transference_general} we know that \eqref{item:local-Hilbert-tranform} is equivalent to the existence of a neighbourhood of the identity $U \subset \G$ such that $\chi_{\widetilde \A}$ defines a Schur multiplier on $S_p(L_2(U))$. By Theorem~A this is equivalent to the existence of a neighbourhood of the identity $V \subset \G$ such that both conditions below hold: 
\begin{align}
    \label{eq:curvature_group_left} T_{h} \partial\widetilde{\A}_{g_1} = T_h \partial\widetilde{\A}_{g_2} &\textrm{ for every }g_1,g_2,h  \in V \textrm{ such that }g_1h,g_2h  \in \partial \A.\\
        \label{eq:curvature_group_right} T_{g} \partial\widetilde{\A}^{h_1} = T_g \partial\widetilde{\A}^{h_2} &\textrm{ for every }g, h_1,h_2 \in V \textrm{ such that }gh_1,gh_2  \in \partial \A.
\end{align}
By the above idenfications \eqref{eq:identification_sheets_Omega}, if we denote by $L_x,R_x\colon \G\to \G$ the left and right multiplication by $x$, these conditions are equivalent to the existence of a neighbourhood of the identity $W$ such that
\begin{align}
    \label{eq:curvature_group_intrinsic_left}  d_{x_1} L_{x_2x_1^{-1}} (T_{x_1} \partial \A) = T_{x_2} \partial \A &\textrm{ for every }x_1,x_2  \in \partial \A \cap W. \\
        \label{eq:curvature_group_intrinsic_right} d_{x_1} R_{x_1^{-1}x_2} (T_{x_1} \partial \A) = T_{x_2} \partial \A &\textrm{ for every }x_1,x_2  \in \partial \A \cap W.
\end{align}
Indeed, taking $x_j = g_j h$ for $j=1,2$ we have \[T_h (g_j^{-1} \partial \A) = d_{x_j} L_{g_j^{-1}}(T_{x_j} \partial \A).\] 
Composing by $(d_{x_2} L_{g_2^{-1}})^{-1} = d_h L_{g_2}$, and using $(d_h L_{g_2}) \circ (d_{x_1} L_{g_1^{-1}})= d_{x_1} L_{x_2x_1^{-1}}$ by the chain rule, we see that  \eqref{eq:curvature_group_left} is equivalent to \eqref{eq:curvature_group_intrinsic_left}. 
The equivalence for right multiplication maps is entirely similar. Next, recalling that $T_e \partial \A = \mathfrak{h}$ the above conditions can be written in the equivalent forms 
\begin{align}
    \label{eq:tangent_space_left}  d_{e} L_x(\mathfrak{h}) =T_x \partial \A&\textrm{ for every }x \in \partial \A \cap W. \\
        \label{eq:tangent_space_right} d_{e} R_{x} (\mathfrak{h}) = T_{x} \partial \A &\textrm{ for every }x \in \partial \A \cap W.
\end{align}
If we remember that $\Ad_x =d_e (R_{x^{-1}} L_x)$, we obtain that this system is equivalent to $d_{e} L_x(\mathfrak{h}) =T_x \partial \A$ and $\Ad_x \mathfrak{h} = \mathfrak{h}$ for every $x \in \partial \A \cap W$. 
Therefore, we have proved that (\ref{item:local-Hilbert-tranform}) at $g_0=e$ is equivalent to the existence of a neighbourhood of the identity $W$ such that $T_x \partial \A = d_e L_x( \mathfrak{h})$ and $\Ad_x \mathfrak{h} = \mathfrak{h}$ for every $x \in \partial \A \cap W$. These conditions clearly hold if $\mathfrak{h}$ is a Lie algebra and $\partial \A$ locally coincides with the exponential of a neighbourhood of $0$ in $\mathfrak{h}$. Conversely, assume $T_x \partial \A = d_e L_x(\mathfrak{h})$ and $\Ad_x(\mathfrak{h}) = \mathfrak{h}$ for every $x \in \partial \A \cap W$. Making $x$ go to the identity element $e$ in the second condition, we deduce that $\ad_\mathrm{X}(\mathfrak h) \subset \mathfrak{h}$ for every $\mathrm{X} \in \mathfrak{h}$. That is, $\mathfrak{h}$ is a Lie subalgebra. By the local uniqueness of a manifold in $\G$ containing $e$ and whose tangent space at $x$ is $d_e L_x(\mathfrak{h})$ (Frobenius' theorem), we deduce that $\partial \A$ is locally the exponential of a neighbourhood of $0$ in $\mathfrak{h}$. This completes the proof.
\end{proof}
\begin{rem} 
\emph{By a partition of the unity argument, the following global form of Theorem~\ref{thm:main_Fourier} holds: if $p \in (1,\infty)\setminus\{2\}$, $\G$ is a connected Lie group and $\A \subset \G$ a relatively compact $\mathcal{C}^1$-domain, then $\chi_\A$ defines a Fourier cb-$L_p$-multiplier if and only if the condition (4) holds for every point $g_0$ in the boundary of $\A$.}
\end{rem}

\begin{proof}[Proof of Corollary~\emph{B2}]

Assertion i) follows since the quotient of a nilpotent Lie algebra remains nilpotent, so the nonnilpotent examples in Theorem \ref{thm:main_Fourier} \eqref{item:three-HTs} cannot happen when $\G$ is nilpotent. Assertions ii) and iii) follow immediately from Lie's classification \cite{zbMATH02681669}: up to isomorphism, there is a unique pair $(\mathfrak{h},\mathfrak{g})$ where $\mathfrak{g}$ is a simple Lie algebra and $\mathfrak{h}$ is a codimension $1$ subalgebra. It is given by  $\mathfrak{g}=\mathfrak{sl}_2$ and $\mathfrak{h}$ the subalgebra of upper-triangular matrices. This completes the proof.
\end{proof}

\subsection{Local Fourier-Schur transference}
\label{sec:proof_transference}
The rest of this paper will be devoted to justify Theorem~\ref{thm:local_transference_general}. We will sometimes consider the Fourier algebra $A(\G)$ of $\G$ \cite{Ey}, that is 
\[A(\G) = \Big\{ g \mapsto \int \phi(gh) \psi(h) \, dh : \phi,\psi \in L_2(\G) \Big\}.\]

A form of the following lemma was proved in \cite[Lemma 1.3]{PRS} 
for $p$ an even integer and $\G$ unimodular, which was enough for the applications there. Here we need a form valid for every $p$ and every locally compact group. 
\begin{lem}\label{lem:local_isomorphic_embedding} Let $V,W\subset \G$ be open sets and $g_0 \in V W^{-1}$. Then, there are a neighbourhood $U$ of $g_0$, a constant $C$, and maps $J_p\colon L_p(\cL \G) \to S_p(L_2(V),L_2(W))$ for $1\leq p\leq \infty$ intertwining Fourier and Herz-Schur multipliers and such that 
      \[ C^{-1} \|x\|_p \leq \|J_p(x)\|_p \leq C \|x\|_p\]
for every $n\geq 1$ and every $x \in M_n \otimes L_p(\cL \G)$ which is Fourier supported in $U$.
\end{lem}
\begin{rem}
\emph{The proof of Lemma~\ref{lem:local_isomorphic_embedding} that we present was kindly communicated to us by \'Eric Ricard. Our original proof was more complicated, but worked whenever $U$ is a relatively compact subset of $V W^{-1}$. The simpler version above is however enough to prove the main implication \eqref{item:localHerzSchur}$\Rightarrow$\eqref{item:localFourier} in Theorem~\ref{thm:local_transference_general}, and by using a partition of the unity argument in the Fourier algebra of $\G$ \cite{Ey}, it is not hard to deduce that this implication holds actually whenever $U$ is a relatively compact subset of $VW^{-1}$.}
\end{rem}
For the reader's convenience, we first prove Lemma~\ref{lem:local_isomorphic_embedding} and Theorem~\ref{thm:local_transference_general} for unimodular groups, and explain in the next paragraph how to modify the definition of Fourier multiplier and the argument for nonunimodular groups.
\begin{proof}[Proof of Lemma~\emph{\ref{lem:local_isomorphic_embedding}} for $\G$ unimodular]  Translating $V$ and $W$, we can assume that the identity belongs to $V$ and $W$ and $g_0=e$. Let $U$ be a neighbourhood of $e$ such that $U \subset V$ and $U^{-1} U \subset W$. Let $\phi = \frac{1}{|U|} \chi_U$ and $\psi = \chi_{U^{-1} U}$, so that $$\int \phi(g h) \psi(h) dh = 1 \quad \mbox{for every $g \in U$.}$$ Consider the map 
\[J_p\colon L_p(\cL \G) \ni x \mapsto \phi^{\frac 1 p} x \psi^{\frac 1 p} \in S_p(L_2(V),L_2(W)),\] 
where we identify $\phi,\psi$ with the operators of multiplication by $\phi,\psi$. The convention is that $0^{\frac 1 \infty} = 0$. We claim that the maps $J_p$ are completely bounded with cb-norm \[ \|J_p\|_{\mathrm{cb}(L_p,S_p)} \leq \|\phi\|_{L_2(\G)}^{\frac 1 p} \|\psi\|_{L_2(\G)}^{\frac 1 p}\] whenever $1 \le p \le \infty$. By interpolation, it is enough to justify the extreme cases $p=1$ and $p=\infty$. The case $p = \infty$ is clear. For the case $p=1$, we factorize $x = x_1 x_2$ so that $J_1(x) = \phi x_1 \cdot x_2 \psi$. Take them so that $\|x\|_1 = \|x_1\|_2 \|x_2\|_2$, and it suffices to show that both factors are bounded in $S_2(L_2(\G)) \simeq L_2(\G \times \G)$ by $\|\phi\|_{L_2(\G)} \|x_1\|_{L_2(\cL \G)}$ and $\|\psi\|_{L_2(\G)} \|x_2\|_{L_2(\cL \G)}$ respectively. Using that \[J_\infty: x \mapsto \Big( \widehat{x}(gh^{-1}) \Big), \] the expected bounds follow from Plancherel theorem $L_2 (\cL \G) \simeq L_2(\G)$. Moreover when $x= \lambda(f)$, the operator $J_p(x)$ has kernel $(\phi(g)^{1/p} f(gh^{-1}) \psi(h)^{1/p})$. Thus it is clear that the map $J_p$ intertwines the Fourier multiplier with symbol $g \mapsto m(g)$ and the Schur multiplier with symbol $(g,h) \mapsto m(gh^{-1})$.
  
The inequality $C^{-1}\|x\|_p \leq \|J_p(x)\|_p$ is a bit more involved. Let $f \in M_n \otimes \mathcal{C}_c(U)$ with $\lambda(f) \in L_p(M_n \otimes \cL \G)$ and assume that $x = \lambda(f)$ by density. Let $q$ be the conjugate exponent of $p$ and $\gamma \in M_n \otimes \mathcal{C}_c(\G)$ with $\lambda(\gamma) \in L_q(M_n \otimes \cL \G)$. Then we have 
\begin{align*} \Tr \otimes \Tr_n \big( J_p(\lambda(f)) J_{q} (\lambda(\gamma))^* \big) & = \Tr_n \int_{\G \times \G} \phi(g) f \gamma^*(gh^{-1})  \psi(h) ds dt\\
    & = \int_\G \Tr_n \big( f(g) \gamma(g)^* \big) \Big[ \int_\G \phi(gh) \psi(h) \, dh \Big] \, dg\\
    & = \int_\G \Tr_n \big( f(g) \gamma(g)^* \big) \, dg = \tau\otimes \Tr_n \big( \lambda(f) \lambda(\gamma)^* \big).
  \end{align*}
In the last line, we used that $\int_\G \phi(gh) \psi(h) \, dh=1$ on $\mathrm{supp} f \subset U$. By H\"older's inequality, we get
  \begin{align*} 
  \big| \tau\otimes \Tr_n( \lambda(f) \lambda(\gamma)^*) \big| & \leq \|J_p(\lambda(f))\|_p \|J_q(\lambda(\gamma)\|_q\\
    & \leq \|J_q\|_{\mathrm{cb}} \|J_p(\lambda(f))\|_p \|\lambda(\gamma)\|_q\\
    & \leq \|J_q\|_{\mathrm{cb}}  \|\lambda(\gamma)\|_q \|J_p(\lambda(f))\|_p.
  \end{align*}
  Taking the sup over $\gamma$, we get $C^{-1}\|x\|_p \leq \|J_p(x)\|_p$ for $C =  \|J_q\|_{\mathrm{cb}}< \infty$.
\end{proof}

\noindent With the same argument as in \cite{PRS}, we deduce:

\begin{proof}[Proof of Theorem~\emph{\ref{thm:local_transference_general}}] The implication \eqref{item:localFourier}$\Rightarrow$\eqref{item:extensionFourier} is easy. Indeed, if \eqref{item:localFourier} holds and $\varphi \in A(\G)$ is supported in $U$ and equal to $1$ on a neighbourhood of $g_0$ (for the construction of $\varphi$, see the proof of Lemma~\ref{lem:local_isomorphic_embedding}), then $T_\varphi$ is completely bounded on $L_p(\cL \G)$ for every $1 \leq p \leq \infty$, and takes values in the space of elements Fourier supported in $U$. In particular, $T_{m\varphi} = T_{m,U}\circ T_\varphi$ is also completely bounded. The implication \eqref{item:extensionFourier}$\Rightarrow$\eqref{item:localHerzSchur} follows from \cite[Theorem 4.2]{CS} which implies that $$(g,h) \in \G \times \G \mapsto m(gh^{-1})\varphi(gh^{-1})$$ defines a completely bounded Fourier $L_p$-multiplier. Thus, we get \eqref{item:localHerzSchur} if $V,W$ are chosen so that $\varphi=1$ on $V W^{-1}$. Finally, \eqref{item:localHerzSchur}$\Rightarrow$\eqref{item:localFourier} follows from Lemma~\ref{lem:local_isomorphic_embedding}.\end{proof}
\subsection{Nonunimodular groups}\label{sec:nonunimodular}
Let $\G$ be an arbitrary locally compact group with modular function $\Delta\colon \G \to \R_+$. Our choice of $\Delta$ is characterized by the following identity for all $f \in \mathcal{C}_c(\G)$ 
\[ \int_\G f(hg) \, dh = \Delta(g)^{-1} \int_\G f(h) \, dh.\]
When $\G$ is not unimodular |that is, $\Delta$ is not the constant $1$ function| the natural weight $\lambda(f)^* \lambda(f) \mapsto \int |f|^2$ on $\cL \G$ is not tracial. Even when $\cL \G$ is semifinite, it is better to work with the general definition of $L_p$ spaces associated to a von Neumann algebra. Several concrete descriptions are possible: Haagerup's original one \cite{Ha}, Kosaki's complex interpolation \cite{Ko}, Connes-Hilsum's \cite{Co,Hi}... see \cite{PX2}. Here we will use the Connes-Hilsum spatial description because we want to rely on some results from \cite{CS,Te}, to which we refer for precise definitions. In that case, $L_p(\cL \G)$ is realized as a space of unbounded operators on $L_2(\G)$.

In \cite{CS}, Caspers and the second-named author defined Fourier $L_p$-multipliers for symbols that ensure that the Fourier multiplier is completely bounded for every $1 \leq p \leq \infty$. Here we extend the definition, allowing to talk about Fourier multipliers for a single $p$, and possibly only bounded. The shortest way to do so properly in this context is by using Terp's Hausdorff-Young inequality \cite{Te}. 

Informally, a typical element of $L_p(\cL \G)$ is of the form $\lambda(f) \Delta^{\frac 1 p}$ |where we are identifying the function $\Delta$ with the densely defined operator of multiplication by $\Delta$ on $L_2(\G)$| for some \emph{suitable} function $f$. Keeping at the informal level, the Fourier multiplier with symbol $m \colon \G\to\C$ should be, whenever it exists, the operator acting as follows \[ \lambda(f)\Delta^{\frac 1 p} \mapsto \lambda(mf)\Delta^{\frac 1 p}.\] Making this definition precise requires some lengthy and unpleasant discussions about domains/cores of unbounded operators, but fortunately we can rely on the results from \cite{Te}, where these discussions have been performed. We shall need to distinguish the cases $p\geq 2$ and $p \leq 2$. Let $q=\frac{p}{p-1}$ be the conjugate exponent of $p$. When $p\geq 2$, the Fourier transform 
\[\mathcal{F}_q \colon L_q(\G) \to L_p(\cL \G)\] is an injective norm $1$ linear map with dense image, where $\mathcal{F}_q(f)$ is defined as a suitable extension of $\lambda(f) \Delta^{1/p}$, see \cite[Theorem 4.5]{Te}. When $p\leq 2$, the adjoint of $\mathcal{F}_p$ gives a norm $1$ injective map with dense image $\overline{\mathcal F}_p\colon L_p(\cL \G) \to L_q(\G)$. 
If $I_q$ denotes the isometry of $L_q(\G)$ defined by $I_q(f)(g) = f(g^{-1}) \Delta(g)^{-1/q}$, we know from \cite[Proposition 1.15]{Te} that every element $x$ of $L_p(\cL \G)$ is a suitable extension of $\lambda(f) \Delta^{1/p}$ for $f = I_q \circ \overline{\mathcal F}_p(x) \in L_q(\G)$. In the particular case $p=2$, these two statements together yield Plancherel's formula : $\mathcal F_2$ is a  unitary. If $p=1$, the image of $\overline{\mathcal{F}}_1$ is the Fourier algebra $A(\G)$, and following standard notation we write 
\begin{equation}\label{eq:def_trace_on_L1}
    \tr(x) = \varphi(e) \quad \textrm{ if } \quad \overline{\mathcal{F}}_1(x)=\varphi.
\end{equation}

\begin{defn}
Let $1 < p < \infty$ and $m \in L_\infty(\G)$. We say that $m$ defines a bounded Fourier $L_p$-multiplier when the condition below holds according to the value of $p \hskip-2pt:$ 
\begin{itemize}
    \item Case $p \geq 2$. The map \[\mathcal{F}_{q}(f) \mapsto \mathcal{F}_{q}(mf)\] (densely defined on $\mathcal{F}_q( L_q(\G))$) extends to a bounded map $T_m$ on $L_p(\cL \G)$.
    
    \vskip3pt
    
    \item Case $p \leq 2$. The multiplication by $m$ preserves the image of $I_q \circ \overline{\mathcal F}_p$ when the map $T_m:x\mapsto (I_q \circ \overline{\mathcal F}_p)^{-1}( m (I_q \circ \overline{\mathcal F}_p(x)))$ is a bounded map on $L_p(\cL \G)$.
\end{itemize}
We say that $m$ defines a completely bounded Fourier $L_p$-multiplier when $m$ defines a bounded Fourier $L_p$-multiplier and the Fourier multiplier $T_m$ is completely bounded. 
\end{defn}
It follows from the above definition that $m$ defines a (completely) bounded $L_p$ multiplier if and only it defines a (completely) bounded $L_q$ multiplier, and in that case 
\[ \tr( T_m(x) y^*) = \tr( x (T_{\overline{m}}(y))^*) \quad \mbox{for all} \quad x \in L_p(\cL \G), y \in L_q(\cL \G).\]

Once we have polished the definition of Fourier $L_p$-multipliers in nonunimodular group von Neumann algebras, we can extend a Cotlar identity from \cite{GPX} to arbitrary locally compact groups. 
\begin{example}\label{ex:cotlar} Let $\G \to \mathrm{Homeo}_+(\R)$ be a continuous action of a connected Lie group. Then, the indicator function $m$ of $\{g \in \G \mid g \cdot 0>0\}$ defines a completely bounded $L_p$ Fourier multiplier on $G$ with completely bounded norm $\leq 2 \max \{p,\frac{p}{p-1}\}$.
\end{example}
\begin{proof}
Let $m$ be the indicator function of $\{g \in \G \mid g \cdot 0>0\}$. It suffices to prove the following implication for every $2 \leq p<\infty$: if $m$ defines a completely bounded Fourier $L_p$-multiplier with norm $\leq C_p$, then it defines a completely bounded Fourier $L_{2p}$-multiplier with norm $\leq 2 C_p$. Indeed, using that $C_2=\|m\|_\infty=1$, we deduce $C_{2^N} \leq 2^N$ for every integer $N$, so by interpolation $C_p \leq 2p$ for all $p\geq 2$. By duality the conclusion also holds for $p\leq 2$.

Let $r$ be the dual exponent of $2p$. Let $f \in \mathcal{C}_c(\G)$ and consider $X =\mathcal{F}_r(f)$ and $Y=\mathcal F_r(mf)$; these are well-defined elements of $L_{2p}(\cL G)$ by \cite{Te}. Then, we claim that the equality below holds 
\begin{equation} \label{Eq-Cotlar}
Y^* Y = T_m(Y^*X) + T_m(Y^*X)^*.
\end{equation}
Indeed, this inequality is equivalent to the almost everywhere equality \[(mf)^* \ast (mf) = m \big( (mf)^* \ast f \big) + \big( m \big( (mf)^* \ast f \big)\big)^*,\] which follows for the fact that $m(g^{-1}) m(g^{-1}h) = m(h) m(g^{-1}) + m(h^{-1})m(g^{-1}h)$ for almost every $g,h \in \G$. If the whole group $\G$ fixes $0$, this is obvious because $m$ is identically $0$. Otherwise, the stabilizer of $0$ is a closed subgroup, so it has measure $0$ and it is enough to justify the equality for $h \cdot 0 \neq 0$. Set $(\alpha,\beta) = (g \cdot 0 ,h \cdot 0)$ and observe that $m(g^{-1}) m(g^{-1}h)=1$ if and only if $\alpha < \min \{0, \beta\}$. Similarly, we have $m(h) m(g^{-1}) = 1$ iff $\alpha < 0 < \beta$ and $m(h^{-1})m(g^{-1}h)=1$ iff $\alpha < \beta < 0$. Therefore the expected identity reduces to the trivial one $\chi_{\alpha < 0 \wedge \beta} = \chi_{\alpha < 0 < \beta} + \chi_{\alpha < \beta < 0}$. This justifies \eqref{Eq-Cotlar}, both sides of which are in $L_{p}(\cL \G)$. Thus, taking the norm and applying the triangle inequality, the hypothesis and H\"older's inequality leads to
\[ \|Y\|_{2p}^2 \leq 2 C_{p} \|X\|_{2p} \|Y\|_{2p}.\]
We deduce $\|Y\|_{2p} \leq 2 C_{p} \|X\|_{2p}$. Since $\mathcal{C}_c(\G)$ is dense in $L_r(\G)$, we obtain that $m$ defines a Fourier $L_{2p}$-multiplier with norm $\leq 2C_p$. A similar argument gives the same bound for the completely bounded norm, which concludes the proof.\end{proof}

\begin{rem}
\emph{The Cotlar-type identity from \cite{GPX} is refined in some cases by \eqref{Eq-Cotlar}.}
\end{rem}

\noindent The following summarizes the properties that we need.

\begin{lem}\label{lem:compression_nonunimodular}
Let $1 \leq p \leq \infty$ and consider functions $\phi,\psi \in L_{2p}(\G)$, which we identify with $($possibly unbounded$)$ multiplication operators on $L_2(\G)$. Then$\hskip1pt :$ 
\begin{itemize}
    \item Given $x\in L_{2p}(\cL \G)$, $x \phi$ is densely defined and closable. In fact, its closure $[x\phi]$ belongs to $S_{2p}(L_2(\G))$ and has $S_{2p}$-norm $\leq \|\phi\|_{L_{2p}(\G)} \|x\|_{L_{2p}(\cL \G)}$.
    \item There exists a bounded linear map $L_p(\cL \G) \to S_p(L_2(\G))$\footnote{That, with a slight abuse of notation, we denote $z \mapsto \psi z \phi$.} sending $y^* x$ to $[y \psi^*]^* [x \phi]$ for every $x,y \in L_{2p}(\cL \G)$. It has norm $\leq \|\phi\|_{L_{2p}(\G)} \|\psi\|_{L_{2p}(\G)}$.
    \item If $q$ denotes the conjugate exponent of $p$, consider $\phi',\psi' \in L_{2q}(\G)$ and $y \in L_{q}(\cL \G)$. Then, we have 
\begin{equation}\label{eq:duality_braket}
\Tr \big( \phi x \psi (\phi' y \psi')^* \big) = \tr( T_m(x) y^*),
\end{equation}
    where $m \in A(\G)$ is the function
    $\displaystyle m(g) = \int_\G (\phi\overline{\phi'})(h) (\psi \overline{\psi'})(g^{-1} h) \, dh$.
\end{itemize}
\end{lem}
\begin{proof}
When $\phi=\psi$ are indicator functions, the first two points were proved in \cite[Proposition 3.3, Theorem 5.2]{CS}. The same argument applies in our case. Let us justify identity \eqref{eq:duality_braket}. First observe that $\phi \overline{\phi'}$ and $\psi\overline{\phi'}$ belong to $L_2(\G)$ by Hölder's inequality, so that $m$ indeed belongs to $A(\G)$. In particular, $m$ defines a completely bounded $L_1$ and $L_\infty$ Fourier multiplier. Thus, it also defines a Fourier $L_p$-multiplier \cite[Definition-Proposition 3.5]{CS}. Therefore, by interpolation \cite[Section 6]{CS} it suffices to prove \eqref{eq:duality_braket} for $p=1$ and $p=\infty$. These two cases are formally equivalent and we just consider $p=\infty$. In that case, $y\in L_1(\cL \G)$ corresponds to an element $f \in A(\G)$ and $\phi' y \psi'$ is the trace class operator with kernel 
\[ \Big( \phi'(g) f(hg^{-1}) \psi'(h) \Big)_{g,h \in \G},\] see \cite[Lemma 3.4]{CS}. By a weak-$*$ density argument, it is enough to prove \eqref{eq:duality_braket} for $x=\lambda(g_0)$ for some $g_0 \in \G$. In that case, $T_m(x) = m(g_0) \lambda(g_0)$ and we obtain $\tr( T_m(x) y^*) = m(g_0) \overline{f(g_0^{-1})}$. We can compute
\begin{eqnarray*}
    \Tr \big( \phi x \psi (\phi' y \psi')^* \big) & = & \Tr \big( \lambda(g_0) \psi \overline{\psi'} y^* \phi \overline{\phi'} \big) \\ [5pt]
    & = & \Tr \Big[ \Big( \psi \overline{\psi'})(g_0^{-1}g) \overline{f(g_0^{-1} g h^{-1} )} (\phi \overline{\phi'})(h) \Big)_{g,h \in \G} \Big]\\
    & = & \int_\G (\psi \overline{\psi'})(g_0^{-1}g) \overline{f(g_0^{-1})} (\phi \overline{\phi'})(g) \, dg = m(g_0) \overline{f(g_0^{-1})}.
\end{eqnarray*}
This justifies the identity \eqref{eq:duality_braket} and completes the proof of the lemma.
\end{proof}

\noindent Lemma~\ref{lem:compression_nonunimodular} allows us to adapt the proof of Lemma~\ref{lem:local_isomorphic_embedding} from the unimodular case.

\begin{proof}[Proof of Lemma~\emph{\ref{lem:local_isomorphic_embedding}}, general case] 
We take $\phi,\psi \in \mathcal{C}_c(\G)$ as in the proof in the unimodular case, and set $m(g) = \int \phi(gh) \psi(h) dh$. By Lemma~\ref{lem:compression_nonunimodular}, we can define completely bounded maps 
\[J_p\colon L_p(\cL \G) \ni x \mapsto \phi^{\frac 1 p} x \psi^{\frac 1 p} \in S_p(L_2(V),L_2(W)),\] 
which intertwine Fourier and Schur multipliers. Now, if $x \in L_p(M_n \otimes \cL \G)$ is Fourier supported in $U$ and $y \in L_q(M_n \otimes \cL \G)$ |for $q$ being the dual exponent of $p$| we get 
\begin{align*}
    \tr(x y^*) & = \tr\big( T_{m}(x) y^*\big) 
    = \Tr \big( J_p(x) J_q (y)^* \big)\\
    & \leq \|J_p(x)\|_{S_p} \|J_q(y)\|_{S_q} \leq \|J_q\|_{\mathrm{cb}} \|y\|_{L_q(\cL \G)} \|J_p(x)\|_{S_p}.
    \end{align*}
    The first line is because $m=1$ on $U$ and $x$ is Fourier supported in $U$, and by  \eqref{eq:duality_braket}. The last line is Hölder's inequality. Taking suprema over $y$ in the unit ball of $L_q(\cL \G)$ gives $\|x\|_{L_p(\cL \G)} \leq \|J_q\|_{\mathrm{cb}} \|J_p(x)\|_{S_p}$.
\end{proof} 
\subsection{The group $\SL_2(\R)$} Consider the symbol
\[m_0 \Big[ \begin{pmatrix} a & b\\c&d  \end{pmatrix} \Big] = \frac12 \Big( 1 + \mathrm{sgn} (ac+bd) \Big).\]
This was identified in \cite{GPX} as the canonical Hilbert transform in $\SL_2(\Z)$. Its complete $L_p$-boundedness follows for $1 < p < \infty$ from a Cotlar-type identity. The same problem in $\SL_2(\R)$ was left open in \cite[Problem A]{GPX}. Now this is solved by condition \eqref{item:Lie-subalg} in Theorem \ref{thm:main_Fourier}, which disproves cb-$L_p$-boundedness for any $p \neq 2$. On the other hand, according to Corollary B2, the map 
\[ m (g) = m \Big[ \begin{pmatrix} a & b\\c&d  \end{pmatrix} \Big] := \frac12 \big( 1 + \mathrm{sgn}(c) \big) = m_0 (gg^\mathrm{t}) \] does define, locally at every point of its boundary, a completely bounded Fourier $L_p$-multiplier for every $1<p<\infty$. But, is it globally $L_p$-bounded? Is it completely $L_p$-bounded as well? We leave these problems open for future attempts.

\subsection{Stratified Lie groups} A Lie algebra $\mathfrak{g}$ is called graded when there exists a finite family of subspaces $\mathrm{W}_1, \mathrm{W}_2, \ldots, \mathrm{W}_N$ of the Lie algebra satisfying conditions below  
\[\mathfrak{g} = \bigoplus_{j = 1}^N \mathrm{W}_j \qquad \mbox{and} \qquad [\mathrm{W}_j, \mathrm{W}_k] \subset \mathrm{W}_{j+k}.\] 
A simply connected Lie group $\G$ is called stratified when its Lie algebra $\mathfrak{g}$ is graded and the first stratum $\mathrm{W}_1$ generates $\mathfrak{g}$ as an algebra. Stratified Lie groups are nilpotent and include, among many other examples, Heisenberg groups. According to Corollary B2 Hilbert transforms are of the form $H \circ \varphi$, for the classical Hilbert transform $H$ and some continuous homomorphism $\varphi: \G \to \R$. A quick look at Theorem \ref{thm:main_Fourier} shows that $\varphi$ corresponds on the Lie algebra with the projection onto any 1-dimensional subspace of the first stratum, since codimension 1 Lie subalgebras are exactly those codimension 1 subspaces leaving out a vector in the first stratum.


\begin{thebibliography}{99}


\bibitem {Bennett} G. Bennett G, Schur multipliers. Duke Math J. \textbf{44} (1977), 603-639.

\bibitem {BF0} M. Bo\.zejko and G. Fendler, Herz–Schur multipliers and completely bounded multipliers of the Fourier algebra of a locally compact group. Boll. Un. Mat. Ital. A \textbf{3} (1984), 297-302.

\bibitem {BF} M. Bo\.zejko and G. Fendler, A note on certain partial sum operators. Banach Center Publications \textbf{73} (2006) 117-225.

\bibitem {DCH} J. de Canni\`ere and U. Haagerup, Multipliers of the Fourier algebras of some simple Lie groups and their discrete subgroups. Amer. J. Math. \textbf{107} (1985), 455-500.

\bibitem {Ca} M. Caspers, A Sobolev estimate for radial $L_p$-multipliers on a class of semi-simple Lie groups. Trans. Amer. Math. Soc. To appear. 

\bibitem {CPPR} M. Caspers, J. Parcet, M. Perrin and \'E. Ricard, Noncommutative de Leuuw theorems. Forum Math. $\Sigma$ \textbf{3} (2015), e21.

\bibitem {CS} M. Caspers and M. de la Salle, Schur and Fourier multipliers of an amenable group acting on non-commutative $L_p$-spaces. Trans. Amer. Math. Soc. \textbf{367} (2015), 6997-7013. 

\bibitem {CJKM} M. Caspers, B. Janssens, A. Krishnaswamy-Usha and L Miaskiwskyi, Local and multilinear noncommutative de Leeuw theorems. Math. Ann. (2023).

\bibitem {CLM} C.Y. Chuah, Z-C. Liu and T. Mei, A Marcinkiewicz multiplier theory for Schur multipliers. ArXiv: 2209.13108. 

\bibitem {CGPT1} J.M. Conde-Alonso, A.M. Gonz\'alez-P\'erez, J. Parcet and E. Tablate, Schur multipliers in Schatten-von Neumann classes. Ann. of Math. \textbf{198} (2023), 1229-1260. 

\bibitem {CGPT2} J.M. Conde-Alonso, A.M. Gonz\'alez-P\'erez, J. Parcet and E. Tablate, A H\"ormander-Mikhlin theorem for higher rank simple Lie groups. ArXiv: 2201.08740. 

\bibitem {Co} A. Connes, On the spatial theory of von Neumann algebras, J. Funct. Anal. 35 (1980), no. 2, 153--164. 

\bibitem {CH} M. Cowling and U. Haagerup, Completely bounded multipliers of the Fourier algebra of a simple Lie group of real rank one. Invent. Math. \textbf{96} (1989), 507-549.



\bibitem {Ey} P. Eymard, L'algèbre de Fourier d'un groupe localement compact. Bull. Soc. Math. France 92 (1964), 181–236.

\bibitem {Fe} C. Fefferman, The multiplier problem for the ball. Ann. of Math. \textbf{94} (1971), 330-336.

\bibitem {Gh} E. Ghys, Groups acting on the circle. Enseign. Math. (2) 47 (2001), no. 3-4, 329-407.

\bibitem {GJP} A.M. Gonz\' alez-P\' erez, M. Junge and J. Parcet. Smooth Fourier multipliers in group algebras via Sobolev dimension. Ann. Sci. \'Ecole Norm. Sup. \textbf{50} (2017), 879-925.

\bibitem {GPX} A.M. Gonz\' alez-P\' erez, J. Parcet and R. Xia. Noncommutative Cotlar identities for groups acting on tree-like structures. ArXiv: 2209.05298.


\bibitem {H} U. Haagerup, An example of a non nuclear $C^*$-algebra, which has the metric approximation property. Invent. Math. \textbf{50} (1979), 279-293.

\bibitem {Ha} U. Haagerup, $L^p$-spaces associated with an arbitrary von Neumann algebra. Algèbres d'opérateurs et leurs applications en physique mathématique (Proc. Colloq., Marseille, 1977), pp. 175–184.

\bibitem {Hi} M. Hilsum, Les espaces $L^p$ d'une algèbre de von Neumann définies par la derivée spatiale, J. Funct. Anal. 40 (1981), no. 2, 151--169.



\bibitem {JMP1} M. Junge, T. Mei and J. Parcet, Smooth Fourier multipliers on group von Neumann algebras. Geom. Funct. Anal. \textbf{24} (2014), 1913-1980. 

\bibitem {JMP2} M. Junge, T. Mei and J. Parcet, Noncommutative Riesz transforms -- Dimension free bounds and Fourier multipliers. J. Eur. Math. Soc. \textbf{20} (2018), 529-595.


\bibitem {Ko} H. Kosaki, Applications of the complex interpolation method to a von Neumann algebra: noncommutative $L^p$-spaces, J. Funct. Anal. 56 (1984), no. 1, 29--78

\bibitem {dLdlS} T. de Laat and M. de la Salle, Approximation properties for noncommutative $L_p$ of high rank lattices and nonembeddability of expanders. J. Reine Angew. Math. \textbf{737} (2018), 46-69.

\bibitem {LdlS} V. Lafforgue and M. de la Salle, Noncommutative $L_p$-spaces without the completely bounded approximation property. Duke. Math. J. \textbf{160} (2011), 71-116.
  
\bibitem {dL} K. de Leeuw, On $L_p$ multipliers. Ann. of Math. \textbf{81} (1965), 364-379.

\bibitem {zbMATH02681669} S. Lie, Theorie der Transformationsgruppen. Dritter (und letzten) {Abschnitt}. {Unter} {Mitwirkung} von {Fr}. {Engel} bearbeitet., Leipzig. {B}. {G}. {Teubner}. {XXVII} + 830 {S}. {{\(8^\circ\)}} (1893).

\bibitem {MR} T. Mei and \'E. Ricard, Free Hilbert transforms. Duke Math. J. \textbf{166} (2017), 2153-2182.

\bibitem {MRX} T. Mei, \'E. Ricard and Q. Xu, A H\"ormander-Mikhlin multiplier theory for free groups and amalgamated free products of von Neumann algebras. Adv. Math. \textbf{403} (2022), 108394.

\bibitem {Mo} G. Mostow, The extensibility of local Lie groups of transformations and groups on surfaces.
Ann. of Math. (2) 52 (1950), 606–636.


\bibitem {NR} S. Neuwirth and \'E. Ricard, Transfer of Fourier multipliers into Schur multipliers and sumsets in a discrete group. Canad. J. Math. \textbf{63} (2011), 1161-1187.

\bibitem {PRS} J. Parcet, \'E. Ricard and M. de la Salle, Fourier multipliers in $S \hskip-1pt L_n(\R)$. Duke Math. J. \textbf{171} (2022), 1235-1297.


\bibitem {PRo} J. Parcet and K. Rogers, Twisted Hilbert transforms vs Kakeya sets of directions. J. Reine Angew. Math. \textbf{710} (2016), 137-172. 

\bibitem {PisAJM} G. Pisier, Multipliers and lacunary sets in non-amenable groups. Amer. J. Math. \textbf{117} (1995), 337-376.

\bibitem {PisAst} G. Pisier, Non-commutative vector valued $L_p$-spaces and completely $p$-summing maps. Asr\'erisque \textbf{247}. Soc. Math. France, 1998. 

\bibitem {PisSim} G. Pisier, Similarity Problems and Completely Bounded Maps. Lecture Notes in Mathematics \textbf{1618}. Springer-Verlag, 2001.

\bibitem {P2} G. Pisier, Introduction to Operator Space Theory. Cambridge University Press, 2003.

\bibitem {PisBAMS} G. Pisier, Grothendieck’s Theorem, past and present. Bull. Amer. Math. Soc. \textbf{49} (2012), 237-323. 

\bibitem {PisS} G. Pisier and D. Shlyakhtenko, Grothendieck’s theorem for operator spaces. Invent. Math. \textbf{150} (2002), 185-217.

\bibitem {PX2} G. Pisier and Q. Xu, Non-commutative Lp-spaces. Handbook of the Geometry of Banach Spaces II (Eds. W.B. Johnson and J. Lindenstrauss) North-Holland (2003), 1459-1517.

\bibitem {PS} D. Potapov and F. Sukochev, Operator-Lipschitz functions in Schatten-von Neumann classes. Acta Math. \textbf{207} (2011), 375-389.

\bibitem {St} E.M. Stein, Harmonic Analysis: Real Variable Methods, Orthogonality, and Oscillatory Integrals. Princeton Math.
Ser. \textbf{43}. Princeton Univ. Press., NJ, 1993.

\bibitem {Te} M. Terp, $L_p$ Fourier transformation on non-unimodular locally compact groups. Adv. Oper. Theory 2 (2017), no.4, 547–583.

\bibitem {MR0120308} J. Tits, Sur une classe de groupes de Lie r\'esolubles, Bull. Soc. Math. Belg. 11 (2) (1959), 100–115.
\end{thebibliography}

\noindent \textbf{Acknowledgement.} The research of JP and ET was partially supported by the Spanish Grant PID2022-141354NB-I00 “Fronteras del An\'alisis Arm\'onico” (MCIN) as well as Severo Ochoa Grant CEX2019-000904-S (ICMAT), funded by MCIN/AEI 10.13039/501100011033. The research of MdlS was partially supported by the Charles Simonyi Endowment at the Institute for Advanced Study, and the ANR project ANCG Project-ANR-19-CE40-0002. ET was supported as well by Spanish Ministry of Universities with a FPU Grant with reference FPU19/00837.

\bibliographystyle{amsplain}

\enlargethispage{1cm}

\vskip5pt

\noindent \textbf{Javier Parcet} \\
\small{Instituto de Ciencias Matem\'aticas, CSIC} \\
\small{Nicolás Cabrera 13-15, 28049, Madrid, Spain} \\
\small{\texttt{parcet@icmat.es}}

\vskip2pt

\noindent \textbf{Mikael de la Salle} \\
\small{Universite Claude Bernard Lyon 1, CNRS, Ecole Centrale de Lyon} \\
\small{INSA Lyon, Université Jean Monnet, ICJ UMR5208, 69622 Villeurbanne, France} \\
\small{\texttt{delasalle@math.univ-lyon1.fr}}

\vskip2pt

\noindent \textbf{Eduardo Tablate} \\
\small{Instituto de Ciencias Matem\'aticas, CSIC} \\
\small{Nicolás Cabrera 13-15, 28049, Madrid, Spain} \\
\small{\texttt{eduardo.tablate@icmat.es}}
\end{document}